\documentclass[11pt]{amsart}
\usepackage[latin1]{inputenc}
\usepackage{amsmath}
\usepackage{amsbsy}
\usepackage{amsfonts}
\usepackage{amssymb}

\usepackage{tikz}
\usetikzlibrary{shapes,fit}
\tikzset{
  >=latex,
  dot/.style={circle,draw,inner sep=0pt,minimum size=1.5ex,color=#1},
  dot/.default=black,
  ddot/.style={circle,draw,inner sep=0pt,minimum size=1.0ex,color=#1},
  ddot/.default=black,
  sdot/.style={circle,fill,inner sep=0pt,minimum size=0.5ex,color=#1},
  sdot/.default=black,
  group/.style={draw=blue,dashed,inner sep=3pt, shape=ellipse},
  group/.default=blue
}

\usepackage{algorithm}
\usepackage{algpseudocode}
\algrenewcommand\algorithmicrequire{\textbf{Input:\ }}
\algrenewcommand\algorithmicensure{\textbf{Output:\ }}

\long\def\nuke#1{\relax}

\title{Quota Trees}
\author{Tad White}
\address{IDA Center for Computing Sciences,
 17100 Science Drive, Bowie, MD 20715-4300}
\date{July 6, 2017}
\email{tad(at)super(dot)org}

\newcommand{\Z}{\mathbf{Z}}
\newcommand{\R}{\mathbf{R}}
\newcommand{\cL}{\mathcal{L}}
\newcommand{\cD}{\mathcal{D}}
\newcommand{\cT}{\mathcal{T}}

\newcommand{\qsymbol}[3]{\left\{\begin{array}{c}#1\\#2\end{array}\right\}_{#3}}

\newcommand{\bs}{\mathbf{s}}
\newcommand{\bq}{\mathbf{q}}
\newcommand\bl{\boldsymbol{\lambda}}
\newcommand\bb{\boldsymbol{\beta}}
\newcommand\bt{\mathbf{t}}
\newcommand\bw{\mathbf{w}}
\newcommand\bp{\boldsymbol{\phi}}
\newcommand\In{\mathbf{in}}

\renewcommand{\d}{\delta}

\newcommand{\inst}[1]{\mathord\to\relax#1}
\newcommand{\outst}[1]{#1\relax\mathord\to}

\newcommand{\diag}{\mathrm{diag}}

\newcommand{\FU}{F_\textrm{used}}
\newcommand{\FS}{F_\textrm{seen}}
\newcommand{\used}{\mathbf{used}}
\newcommand{\seen}{\mathbf{seen}}

\let\df\textit

\newtheorem{theorem}{Theorem}
\newtheorem{corollary}[theorem]{Corollary}

\begin{document}

\sloppy
\begin{abstract}
We introduce the notion of quota trees in directed graphs. Given a
nonnegative integer ``quota'' for each vertex of a directed multigraph
$G$, a quota tree is an immersed rooted tree which hits each vertex of
$G$ the prescribed number of times. When the quotas are all one, the
tree is actually embedded and we recover the usual notion of a
spanning arborescence (directed spanning tree). The usual algorithms
which produce spanning arborescences with various properties typically
have (sometimes more complicated) ``quota'' analogues.

Our original motivation for studying quota trees was the problem of
characterizing the sizes of the Myhill-Nerode equivalence classes in a
connected deterministic finite-state automaton recognizing a given
regular language. We show that the obstruction to realizing a given
set of M-N class sizes is precisely the existence of a suitable quota
tree.

In this paper we develop the basic theory of quota trees.
We give necessary and sufficient conditions for the
existence of a quota tree (or forest) over a given directed graph with
specified quotas, solving the M-N class size problem as a special
case.  We discuss some potential applications of quota trees and
forests, and connect them to the $k$ lightest paths problem.  We give
two proofs of the main theorem: one based on an algorithmic loop
invariant, and one based on direct enumeration of quota trees. For the
latter, we use Lagrange inversion to derive a formula which vastly
generalizes both the matrix-tree theorem and Cayley's formula for
counting labeled trees.  We give an efficient algorithm to sample
uniformly from the set of forests with given quotas, as well as a
generalization of Edmonds' algorithm for computing a minimum-weight
quota forest.
\end{abstract}
\fussy

\maketitle

\section{Motivation and definitions}

A recently proposed scheme in the area of private information
retrieval \cite{fink2017streaming}
rests in part on the ability to construct
arbitrarily complex deterministic finite automata (DFAs) recognizing a
regular language $\cL$.  While the theory of simplifying, or
minimizing, a finite-state automaton is well known, the inverse
problem of ``complicating'' a DFA leads to interesting questions about
the structure of the set of DFAs recognizing $\cL$.

The Myhill-Nerode theorem implies the existence of a unique minimal
DFA $\cD_\cL$ which recognizes $\cL$. $\cD_\cL$ is a quotient of any
connected%
\footnote{Think of a DFA as a graph $G$ having an initial node and
labeled edges coming out of each node; the DFA is \df{connected}
if any node in $G$ can be reached from the initial node.}
DFA $\cD$ recognizing $\cL$;
that is, the states of $\cD$ can be grouped into equivalence classes,
with one class for each state of $\cD_\cL$, such that the transitions
in $\cD$ are coherent with respect to these classes. So in order to
understand the set of connected DFAs which can recognize $\cL$, one
wants to know what sizes these equivalence classes can take, and to
have an effective algorithm for constructing a connected DFA with
given class sizes.  (Connectedness is the key issue here; if $\cD$ is
allowed to have unreachable nodes, then there is no constraint on the
sizes other than positivity.)

The problem turns out to reduce to a very natural graph search problem.%
\footnote{%
Throughout this paper, we will often use the term ``graph'' to mean
what is usually called a directed multigraph; that is, edges are
directed, and both loops and multiple edges are allowed. We will
frequently encounter directed trees, with edges directed away from the
root; these are typically called
\textit{(out-)arborescences} in the literature. Accordingly, forests
of out-directed trees should be called something like \textit{silvations}.
But we will stubbornly just use ``trees'' and ``forests.''}
In particular, it turns out that a connected DFA with specified
Myhill-Nerode class sizes exists iff one can construct a
directed tree $T$, together with an immersion $f:T\to G$, such
that the sizes of the vertex preimages match the desired equivalence
class sizes; we call $T$ a ``quota tree.'' When it exists, a suitable $T$
can be found via a simple modification of standard graph traversal in
which vertices are visited multiple times, according to the class sizes;
$T$ records the traversal just as an ordinary graph
search is recorded via a spanning tree.
$T$ can then be extended (in many ways) to a DFA by adding missing
transitions.

It is easy to interpret this type of graph search in applications other
than automata; the theory expresses itself most naturally in a broader
context. In section \ref{sec:examples} we describe some
scenarios in which quota trees arise naturally; these illustrate some
quota versions of standard spanning tree optimization problems.
In section \ref{sec:quotatrees} we formally define quota trees and
forests and state the main results. 
Section \ref{sec:quotasearch} introduces the corresponding variant of
graph search, called \textit{quota search}. In section
\ref{sec:enougharrows} we prove the ``enough arrows'' theorem, which
gives necessary and sufficient conditions for the existence of quota
trees (or forests) with specified quotas. In section
\ref{sec:applications} we discuss some applications, particularly
DFAs and the $k$ lightest path problem. In section
\ref{sec:enumeration} we address the problem of enumerating quota trees;
our primary tool is the multivariate Lagrange inversion
formula. The results of
this section give a much more precise version of the ``enough arrows''
theorem, which vastly generalizes both the matrix-tree theorem and
Cayley's formula for counting labeled trees. In section \ref{sec:generation}
we strengthen the enumeration results to sample uniformly from the set
of trees (or forests) with given quotas. In section \ref{sec:mqf} we
give an algorithm for finding minimal-weight quota forests.
Finally, in section
\ref{sec:further-work}, we identify a few areas for further research.

\section{Examples}
\label{sec:examples}
Before giving formal definitions, we present a few scenarios in which
quota trees arise naturally, so that the reader can choose a comfortable
motivating context for the remainder of the paper.

\subsubsection*{A coupon game}
A dealer has a supply of coupon books of various types; all books of a
given type are identical. Each coupon in a book allows you to
buy another coupon
book at a particular price.  (For example, it might be that in an $A$
book, coupon 1 is for another $A$ book at \$5, coupons 2 and 3 are for
$B$ books at \$2 and \$3 respectively, and coupon 4 is good for a free
$D$ book.)  You're given one
or more coupon books to start with; you win if you can collect all of
the dealer's coupon books (and you'd like to do so as cheaply as
possible.)  You know how many books of each type the dealer has, and
what coupons are in what types of book. Is it possible to collect all
the coupons? If so, in how many different ways, and what is the
minimum cost?

\subsubsection*{Network configuration}
You have a supply of network devices of various types; all devices of
a given type are identical. Each type has a single input port and
several output ports, each of which can talk to a specific type of
device.  (For example, an $A$ device might have an $A$ port, two $B$
ports and a $D$ port, a $B$ device might have no output ports, and so
on.) You would like to connect all of your devices together so that a
message from one particular device can then be propagated to all of
the other devices. Is this possible? If so, in how many ways, and
what configuration minimizes the number of intermediate devices
on each path? When there is only one device of each type, this is a
spanning tree problem.

\subsubsection*{$k$ lightest paths}
Given a directed graph $G$, with nonnegative weights on the edges, and
an integer $k\ge1$, compute the $k$ lightest paths from one or more
given source nodes to each vertex in $G$. This can be interpreted as
a minimum quota tree problem.

\subsubsection*{Tree coloring}
Many tree-coloring problems are naturally interpreted as quota-tree
questions. For example, suppose we have $n$ colors, and a subset
$S_i\subset [n]$ for each $i\in[n]$. How many ways can we color
a rooted tree such that $q_i$ nodes have color $i$, and such that
the children of a color-$i$ node get distinct colors selected
from $S_i$? (For example, for two colors, if there are no restrictions
we get Narayana numbers; if blue nodes can only have red children we
get Motzkin coefficients. See section \ref{sec:enumeration} for more
examples.)

\section{Quota trees}
\label{sec:quotatrees}

By a directed multigraph we will mean a tuple $(V,E,i,t)$ where $V$ is
a set of \df{vertices}, $E$ is a set of \df{edges}, and $i:E\to V$ and
$t:E\to V$ return the \df{initial} and \df{terminal} vertices of each
edge. Edges are oriented; we say an edge $e$ goes \df{from} $i(e)$
\df{to} $t(e)$. We may abuse notation by writing $v\to w$ to mean
there is an edge in $G$ from $v$ to $w$, but as $G$ is a multigraph
there may be other edges as well.  In particular, loops are allowed
(that is, one may have $t(e)=i(e)$ for some edges $e$) and the edges
from $v$ to $w$ are all distinguishable (or ``labeled'') as they are
distinct elements of $E$.)

A mapping $f:G\to H$ of multigraphs sends vertices to vertices, edges
to edges, and respects $i$ and $t$: thus, if $e$ is an edge from $v$
to $w$ in $G$, then $f(e)$ is an edge in $H$ from $f(v)$ to $f(w)$.

Define the \df{instar} (resp.~\df{outstar}) of a vertex $v$ to be the
set of incoming (resp.~outgoing) edges at $v$:
$$\inst v = \{e \mid t(e)=v\};\qquad \outst v = \{e \mid i(e)=v\}$$
We say a map $f:G\to H$ is an \df{out-immersion}, or simply an
\df{immersion}, if it maps $\outst v$ injectively into $\outst{f(v)}$.
We define a \df{cusp} in $G$ under $f$ to be a pair
of edges $e_1\ne e_2\in \outst v$ with $f(e_1)=f(e_2)$; thus $f$
is an immersion iff $G$ has no cusps under $f$.

A \df{quota} is a nonnegative-valued function $q:V(G)\to\Z$. A
\df{quota tree} with root $*\in V(G)$ and quota $q$ is an immersion
$f:T\to G$ where $(T,\tilde *)$ is a rooted tree, $f(\tilde *)=*$, and
$|f^{-1}(v)|=q(v)$ for all $v\in V(G)$.  Note that if $q(v)\le1$ for
all $v\in V(G)$, then the map $f$ is actually an embedding, and if
$q(v)$ is identically $1$, the image $f(T)$ is a (rooted, directed)
spanning tree of $G$.

Finally, a \df{quota forest with start portfolio} $s:V(G)\to\Z$
is a (disjoint) union of quota trees
$F=\{T_{v,i}\mid v\in V(G),1\le i\le s(v)\}$
such that $T_{v,i}$ is rooted at $v$. The forest also immerses into
$G$; the quota it achieves is the sum of the quotas of the
component forests.  Note that we treat all the roots as
distinguishable: if a forest contains two or more
non-isomorphic quota trees with roots mapping to the same vertex
of $G$, permuting those trees gives a different quota forest.
We will refer to a forest with quota $q$ and start portfolio $s$
as a $(G,q,s)$-forest (or $(G,q,s)$-tree, if $||s||_1=1$).

A graph with quotas, together with both an example and a nonexample
of quota trees, appears in Figure \ref{fig:QTexample}.
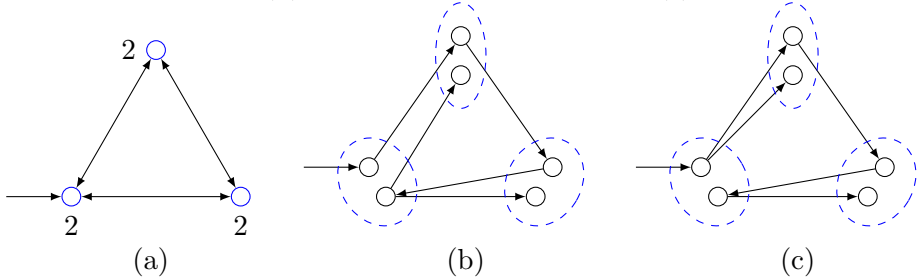
\begin{figure}
\caption{A digraph (a) with quotas and a single-vertex
start portfolio; (b) is a valid quota tree, while (c) is not.}
\label{fig:QTexample}
\begin{tabular}{ccc}
\raisebox{-1ex}{\begin{tikzpicture}
[scale=1.3]
\node at (210:1) [dot=blue] (A) {}; \node [below] at (A.south) {$2$};
\node at (90:1)  [dot=blue] (B) {}; \node [left] at  (B.west) {$2$};
\node at (330:1) [dot=blue] (C) {}; \node [below] at (C.south) {$2$};

\node [left of=A] {} edge [->] (A);

\draw[<->] (A) -- (C);
\draw[<->] (A) -- (B);
\draw[<->] (B) -- (C);
\end{tikzpicture}}
&
\raisebox{0ex}{\begin{tikzpicture}
[scale=1.3]
\node at (200:1)  [dot] (A1) {};
\node at (220:1)  [dot] (A2) {};
\node at (90:1)   [dot] (B1) {};
\node at (90:0.6) [dot] (B2) {};
\node at (340:1)  [dot] (C1) {};
\node at (320:1)  [dot] (C2) {};

\node [left of=A1] {} edge [->] (A1);

\draw[->] (A1) -- (B1);
\draw[->] (B1) -- (C1);
\draw[->] (C1) -- (A2);
\draw[->] (A2) -- (C2);
\draw[->] (A2) -- (B2);

\node[group, rotate=30, fit=(A1) (A2)] {};
\node[group,            fit=(B1) (B2)] {};
\node[group, rotate=-30, fit=(C1) (C2)] {};

\end{tikzpicture}}
&
\raisebox{0ex}{\begin{tikzpicture}
[scale=1.3]
\node at (200:1)  [dot] (A1) {};
\node at (220:1)  [dot] (A2) {};
\node at (90:1)   [dot] (B1) {};
\node at (90:0.6) [dot] (B2) {};
\node at (340:1)  [dot] (C1) {};
\node at (320:1)  [dot] (C2) {};

\node [left of=A1] {} edge [->] (A1);

\draw[->] (A1) -- (B1);
\draw[->] (B1) -- (C1);
\draw[->] (C1) -- (A2);
\draw[->] (A2) -- (C2);
\draw[->] (A1) -- (B2);

\node[group, rotate=30, fit=(A1) (A2)] {};
\node[group,            fit=(B1) (B2)] {};
\node[group, rotate=-30, fit=(C1) (C2)] {};

\end{tikzpicture}}
\\
\qquad(a)&\qquad(b)&\qquad(c)
\end{tabular}

\end{figure}

\subsection*{Out-coverings}
One can think of a spanning tree $G$ as ``living in'' $G$,
but a more natural home for a quota tree is actually
a covering space of $G$, which we will now describe.
We will say a map $\pi:G\to H$ is an
\df{out-covering} if $\pi(\outst{v})$ maps bijectively onto
$\outst{\pi(v)}$ for all $v$ in $V(G)$. In this situation,
given an (out-)immersed tree $f:T \to H$ with root $w\in H$, and a preimage
$v\in \pi^{-1}(w)$, there is a unique lift $\tilde f:T\to G$
with root $v$; the (right) inverse of the operation
$f\mapsto\tilde f$ is given by $f\mapsto \pi\circ f$.

As with topological spaces, we can define a universal out-cover by
considering paths from a distinguished vertex. A (finite directed)
\df{path} in $G$ from $*$ is a sequence $\{e_i\mid 1\le i\le l\}$ of
directed edges of $G$, with $i(e_1)={*}$ and $t(e_i)=i(e_{i+1})$.  We
define the \df{universal out-cover} of $(G,{*})$ to be the directed
graph $(\tilde G,\tilde {*})$ whose vertices are the finite directed
paths from $*$, having an edge (labeled $e_i$) from $e_1\cdots
e_{i-1}$ to $e_1\cdots e_i$. It's easy to see that $\tilde G$ is a
(generally infinite) rooted tree, in which the root $\tilde *$
corresponds to the length-zero path in $G$ from $*$. The natural map
$\pi:\tilde G\to G$ taking a directed path to its endpoint in $G$ is
an immersion.  Note that the in-degree of each vertex $\tilde
v\in\tilde G$ is one; the out-degree of $\tilde v$ is the same as the
out-degree of $\pi(\tilde v)$. (In particular, if $G$ is a DFA over an
alphabet $\Sigma$, then $\tilde G$ is a regular tree, directed outward
from the root, with each vertex having out-degree $|\Sigma|$.)

With this setup, it is easy to see that if $f:(T,t)\to (G,{*})$ is an
immersion of a rooted directed tree into $G$, then $f$ can be lifted
uniquely to a map $\tilde f:(T,t)\to (\tilde G,\tilde {*})$ such that
$f=\pi\circ\tilde f$.%
\footnote{There is a larger ``universal cover'' that appears in the
  literature (see for example \cite{norris1995universal}), based on
  paths whose edges which need not be coherently oriented. This is
  essentially the topological universal cover of $G$ (see
  \cite{Stallings1983}), constructed by ignoring orientations, which
  also has the same universal lifting property for immersions of
  rooted directed trees.  However, the universal out-cover is the
  smallest space which has this property, and so is the ``natural''
  home for quota trees.  We note that Yamashita and Kaneda, in their
  study of computing in anonymous networks, referred to the
  universal out-cover $(\tilde G,\tilde {*})$ as the \df{view}
  of $\tilde *$ within the topological universal cover (see
  \cite{Yamashita:1988:CAN:62546.62568}.)}
The map $\tilde f$ is injective, so we can view $T$ as sitting inside
of $\tilde G$.

\section{Quota search}
\label{sec:quotasearch}

The problems in section \ref{sec:examples}, as well as the original
problem of computing possible Myhill-Nerode class sizes, correspond
to a variant of graph search in which we are given a positive
``quota'' $q(v)$ for each $v\in V(G)$, and we wish to visit each
vertex $v$ exactly $q(v)$ times. (When $q(v)=1$ for all $v$, this is
ordinary graph traversal.)%
\footnote{Setting $q(v)=0$ for any particular $v$ is legal; it
 essentially amounts to working in the induced graph $G-\{v\}$.}
We refer to this goal as \df{quota search}.

We assume familiarity with standard graph traversal as described,
for example, in \cite[ch.~22]{cormen2009introduction}, to which we
will make some modifications.
Given a directed graph $G$ and a set $S$ of start vertices,
generic graph search keeps track of discovered but unprocessed
vertices in a generic priority queue. As we will be
dealing with multigraphs, and visiting vertices multiple times, we
will need to be more careful to distinguish between an edge from
$u$ to $v$ and the pair $(u,v)$ itself; indeed, it is much easier
to describe quota search succinctly by considering edges rather than
vertices. So our algorithm encodes the search forest $F$
via a predecessor function $\pi:E(F)\to E(F)$, rather than the more
usual $\pi:V(G)\to V(G)$.  Accordingly, we replace the usual
\textsc{VisitVertex} procedure with an analogous \textsc{UseEdge},
which inserts an edge taken from the queue into the search forest.

Recall that the quota forest $F$ does not actually live in $G$, so
we must distinguish between an edge $\tilde e$ in $F$ and its image
$e=f(\tilde e)$ under the immersion $f:F\to G$, whose construction
will be implicit. An edge $\tilde e$ in the queue should be thought
of as living in the universal cover $\tilde G$, not in $G$.

Instead of coloring vertices \textsc{Black} or \textsc{White}
according to whether or not they have been visited, we
keep track of the number of remaining visits to a vertex $v$. Thus,
\textsc{Black} and \textsc{White} correspond to quotas of $0$ and $1$
respectively.

In ordinary graph search, we gain nothing by repeating a search from
the same start point, but allowing repeated starts even from the same
node can be useful if we need to arrive at vertices multiple times.
As described in section \ref{sec:quotasearch}, we replace the set $S$
of start vertices with a nonnegative start portfolio $s:V(G)\to \Z$;
$s(v)$ is the number of times a search can be started
from a given vertex. (Thus, if we're doing a single search from one
particular vertex $w$, we set $s(v)=1$ if $v=w$ and $0$ otherwise.)

Finally, it is useful to distinguish two natural variants of search.
In the ``exact'' version of quota search, we want our search forest to
contain exactly $s(v)$ trees with root $v$. (This corresponds, in the
coupon-game scenario, to requiring that every coupon in the initial
collection be used up.) In the ``at-most'' version, the search forest
may contain \textit{at most} $s(v)$ trees with root $v$; that is, we
don't need to use all of our coupons. The two versions are closely
related:

\begin{theorem}[exact vs.~at most solvability]
A triple $(G,q,s)$ admits an exact quota forest iff it admits an at-most
 quota forest and $q(v)\ge s(v)$ for all $v\in V(G)$.
\end{theorem}
\begin{proof} Since an exact forest actually solves the
at-most problem, and clearly requires $q(v)\ge s(v)$ for all $v\in V(G)$,
one direction is trivial. On the other hand,
if we have an at-most quota forest $F$ with fewer than $s(v)$ trees
rooted at lifts of $v$, we can simply cut off some of the $q(v)$
occurrences in $F$ of lifts of $v$ from their parents, making them
roots of new trees. This works as long as $q(v)\ge s(v)$.
\end{proof}

Both the exact and at-most versions of quota search can be handled
with a single meta-algorithm. In both cases we initialize $Q$ with
(sentinel edges corresponding to) the start portfolio. In order to
implement \textsc{ExactQuotaSearch}, we simply arrange for
\textsc{QueueExtract} to return the start portfolio first; to
implement \textsc{AtMostQuotaSearch}, we drop that restriction, in
which case the number of new trees created, and what their roots are,
will depend on the particular queue extraction algorithm.

We capture the resulting generic quota search
meta-algorithm as Algorithm \ref{alg:genericquotasearch}. It succeeds
if it ends with all quotas reduced to zero. The ``enough arrows''
theorem will characterize triples $(G,q,s)$ such that a quota forest
exists (in which case the algorithm is guaranteed to succeed for any
specialization of \textsc{QueueExtract}.)

\begin{algorithm}
\caption{Generic quota search}
\label{alg:genericquotasearch}

\begin{algorithmic}[5]
\Function {NewEdge}{$\tilde e$,$e'$}
  \Require $e'\in E(G)$; $\tilde e\in E(F)$ with $f(t(\tilde e))=i(e')$
  \Ensure a new edge $\tilde e'$ with $f(\tilde e')=e'$, $\pi(\tilde e')=\tilde e$
\EndFunction
\Statex
\Function {NewSentinelEdge}{$v$}
  \Require $v\in V(G)$ \Comment $v$ will be the root of a tree in $F$
  \Ensure a new sentinel edge $\tilde e$ with $f(\tilde e)=NULL$, $f(t(\tilde e))=v$
\EndFunction
\Statex
\Procedure {UseEdge}{$\tilde e$}
  \Require an edge $\tilde e$ such that $v=f(t(\tilde e))$ satisfies $q(v)>0$
  \Ensure Add $\tilde e$ to $F$; this updates $F$ and $q(v)$ and adds
          $\outst{t(\tilde e)}$ to $Q$
  \State $F=F\cup\{\tilde e\}$
  \State $q(v)\gets q(v)-1$
  \For {$e' \in \outst v$}
    \Comment in practice, skip $e'$ if $q(t(e'))$ is already zero
    \State \textsc{QueueInsert}$(Q,\textsc{NewEdge}(\tilde e,e'))$
  \EndFor
\EndProcedure
\Statex
\Function {GenericQuotaSearch} {$G$,$q$,$s$}
\Require $G$ a directed graph; $q$ and $s$ are nonnegative functions on $V(G)$
\Ensure quota forest $F$, predecessor map $\pi:E(F)\to E(F)$, and
  immersion $f:F\to G$
  \State $Q,F\gets\emptyset$
  \For {$v\in V(G),  k\in\{1,\ldots s(v)\}$}
    \State $\textsc{QueueInsert}(Q,\textsc{NewSentinelEdge}(v))$
  \EndFor
\item[main loop:]
  \While {$Q$ is nonempty } 
    \State $\tilde e\gets \textsc{QueueExtract}(Q)$
    \If {$q(f(t(\tilde e)))>0$}
      \State $\textsc{UseEdge}(\tilde e)$
    \EndIf
  \EndWhile
  \State \textbf{return} $\pi,f$ {\textbf{unless}\textrm{ all $q(v)$'s are zero}}
  \Comment else fail
\EndFunction

\end{algorithmic}
\end{algorithm}

\subsubsection*{Algorithm success and achievable parameters}

Whenever \textsc{UseEdge} is called, $q(v)$ is the number of
remaining required visits to $v$. Thus the algorithm succeeds
(i.e.~visits all vertices the required number of times)
if and only if, upon
termination, $q(v)=0$ for all $v\in V(G)$.  It turns out that, in
contrast with ordinary graph search, success is not possible for all
pairs $(q,s)$.  We will call $(G,q,s)$ \df{achievable} if some
(and, it turns out, any) quota search in $G$ with start portfolio $s$
achieves the quotas $q$. (It is easy to see that achievability does not
depend on whether we are talking about ``exact'' or ``at most'' quota
search.)  
The ``enough arrows'' theorem in the next section precisely characterizes
the achievable parameters.

\subsubsection*{Quota search viewed in $\tilde G$}
One way to think about quota search is that we replace each vertex $v$
with a supply of copies of itself; when we ``visit'' $v$, we're actually
visiting a fresh copy.  When the start portfolio is [a single copy of]
a single vertex $*$, this allows us to describe quota search as
occurring in the forward universal cover $\tilde G$ of
$G$. Specifically, we do ordinary graph search in $(\tilde G,\tilde
*)$, but only visit a vertex $\tilde v$ provided $q(v)>0$, where
$v=\pi(\tilde v)$, in which case we decrement $q(v)$.  Finally, if the
start portfolio $s$ is a multiset of vertices, we effectively work in
the disjoint union of $s(v)$ copies of $(\tilde G,\tilde v)$ for
all $v$. Whether the search trees are built sequentially, or at the
same time, is controlled by the order in which \textsc{QueueExtract}
selects edges for consideration.

\subsubsection*{Optimization problems}
As with ordinary graph search, the versatility of this meta-algorithm
comes from the variety of ways of choosing which element to extract
from $Q$ at each step. By specializing $Q$ to be a FIFO queue,
a LIFO stack, or a more general priority queue results in quota-search
we obtain quota variants of algorithms such as breadth-first search,
depth-first search, or Dijkstra's algorithm.

If we are optimizing an objective function which depends only on the
forest $F$, but not the particular traversal of $F$, then the data
associated with an edge $\tilde e$ in the queue $Q$ may only depend on
the unique path to $\tilde e$ in $F$; we will call such data
\df{intrinsic}. For example, if the edges of $G$ have weights, it is
natural to consider the ``minimum quota forest'' problem, a
generalization of the minimum spanning tree problem in which we wish
to minimize the sum of the weights of the edges in a quota forest with
the given start portfolio and quotas. In this case we take the key for
an edge $\tilde e$ in $Q$ to be the weight of its image $e=f(\tilde
e)$ in $G$. Similarly, a quota version of Dijkstra's algorithm is
obtained by taking the key to be the sum of the weights in the path to
$\tilde e$ in the search forest; see section \ref{sec:applications}.
In both cases the keys are intrinsic.

It may be tempting, knowing that a vertex will be visited $q(v)$
times, to assign the $k$-th visit to a vertex a cost which depends on
$k$. However, this is not intrinsic: different traversals of the same
forest could then result in different tree costs.  But it would be
perfectly legal to assign edge $\tilde e$ a cost which depends on the
number of visits to $t(\tilde e)$ (or any other nodes) on the
path in $F$ to $\tilde e$.

Of course, not all graph optimization problems are solvable via graph
search.  For instance, a very natural problem is to find a
minimum-weight quota tree (or forest) given weights on the edges of
$G$; here we must emphasize that we really mean quota arborescence (or
branching.)  When the quotas are at most $1$, this is just the minimum
arborescence (or branching) problem.  An algorithm for solving this
problem has been given by Edmonds \cite{edmonds1967optimum} and
others. Rather than accreting a tree via graph search, it iterates
through a sequence of putative solutions. Edmonds' algorithm adapts
beautifully to find minimum quota trees (and, in particular, find the
minimum-cost solution to the coupon collecting problem in section
\ref{sec:examples}.) We discuss minimum-weight quota trees in
section \ref{sec:mqf}.

\subsubsection*{Relaxation}
Many natural priority queue keys have a property which allows us to
maintain a smaller queue. As noted previously, an intrinsic cost
associated to an edge $\tilde e$ in $Q$ is some function $c(\tilde p)$
of the unique path $\tilde p=\tilde e_1\cdots \tilde e_k=\tilde e$
in the quota forest from the root to $\tilde e$.
We say $c$ is \df{append-monotonic} if key order is invariant under
appending a common path: that is, if we have two paths
$\tilde p_1$ and $\tilde p_2$ satisfying
$c(\tilde p_1)\le c(\tilde p_2)$, and both ending at lifts of a
common vertex $v$, then for any path $p_3$ in $G$ starting at $v$,
then
$$
c(\tilde p_1\tilde p_3)\le c(\tilde p_2\tilde p_3).%
\footnote{Here the two $\tilde p_3$'s are strictly different, since
they represent the lifts of $p_3$ to the endpoints of $\tilde p_1$
and $\tilde p_2$ respectively.}
$$

If $f$ is append-monotonic, we know the best extensions of paths
will be extensions of best paths. So we can just keep track of
the $q(v)$ best paths to each vertex $v$.
This is the quota-search analogue of what is called
\textit{relaxation} in ordinary graph search (see
\cite[Ch.~24]{cormen2009introduction}): namely, when we arrive at
a previously seen vertex via a new path, we can keep the better of the
two paths and discard the other. In generic quota search, we might
handle this with a min-max queue of size $q(v)$ at each node $v$,
in which case a generic implementation of \textsc{QueueExtract} via
two stages of binary heaps would
take $\lg V+\lg q(v)$ operations.

\subsubsection*{Complexity analysis}
In the generic version of quota search, we visit each vertex $v$
$q(v)$ times, doing one queue extraction and $|\outst v|$ insertions.
So the number of insertions and extractions (and space) required is
$\sum_v q(v) Adj(v)$ where $Adj(v)=|\outst v|+1$.
When
$q(v)=1$ for all $v$, and $Q$ is a simple queue or stack (so that
insertions and extractions can be done in constant time), note that
this reduces to $O(V+E)$, the complexity of ordinary graph search.

If $Q$ is a priority queue, this leads to a complexity of
$$O\left( \sum_v q(v) Adj(v) (\lg \sum_v q(v) Adj(v))\right)$$
operations if binary heaps are used.
If the queue keys are append-monotonic, we can apply relaxation
as above, reducing the work to
$$O\left( \sum_v q(v) Adj(v) (\lg V+\lg q(v))\right).$$
(This reduces to $O(E\lg V)$ when the quotas are identically 1.)
As usual, more sophisticated heap structures can provide further
asymptotic improvement.

\section{The Enough Arrows theorem}
\label{sec:enougharrows}
In this section, we identify two conditions which the data $(G,q,s)$
must satisfy in order for quota search to succeed; one is global, the
other is local.  We show that these conditions are in fact sufficient:
there exists a quota forest meeting the specified quotas if and only
if these conditions hold.  (In section \ref{sec:enumeration} we will
give an independent proof based on direct enumeration of quota forests.)

\begin{description}
\item[Global] $(G,q,s)$ is \emph{connected} if, for
every node $v$ with $q(v)>0$, there exists a node $u$ with $s(u)>0$ and a path
in $G$ from $u$ to $v$. Note this only depends on the support of $q$ and $s$.
\item[Local] $(G,q,s)$ has \emph{enough arrows} if the inequality
\begin{equation}
\label{eq:enougharrows}
s(w) + \In(w) \ge q(w)
\end{equation}
holds for each $w\in V(G)$, where
$\In(w):=\sum_v q(v)m_{vw}$.
\end{description}

We remark that the enough arrows condition can be written as
$$\bs + \bq M \ge \bq,$$
where $\bq$ and $\bs$ are the vectors of values of $q$ and $s$
respectively, and $M$ is the adjacency matrix of $G$.

Connectivity is clearly necessary in order to achieve even one visit
to every vertex with positive quota. To
see why having enough arrows is necessary, note that each visit to
node $w$ arises either by starting at $w$, or by following an edge
from another node $v$.  We visit node $v$ $q(v)$ times; each time, we
have $m_{vw}$ edges we can potentially follow to node $w$.  Thus the
maximum number of arrivals at node $w$ is the left-hand side of
(\ref{eq:enougharrows}), which must be at least $q(w)$.

A note on terminology: especially in the context of automata, directed
graphs are typically drawn with arrows representing both transitions
and initial states.  The left-hand side of the inequality
(\ref{eq:enougharrows}) counts the maximum number of arrows that can
be drawn into each class (see figure \ref{fig:QTexample});
the right-hand side represents the number
of targets that need to be hit by these arrows.

\begin{theorem}[enough arrows]
\label{thm:maintheorem}
With the notation above, generic at-most quota search in $G$ with start
portfolio $s$ will achieve the quotas $q$ if and only if $(G,q,s)$ is
connected and has enough arrows.
\end{theorem}

\begin{proof}
We have already argued the necessity of these conditions. The converse
is essentially by induction on $\sum_{v,w} q(v) m_{vw}$, and will follow
from the fact that connectivity and having enough arrows are invariant
under the main loop. Connectivity is automatically preserved.

So suppose we have enough arrows entering the main loop.
At each iteration, the
queue $Q$ represents an effective ``at most'' start portfolio;
so let $s(v)$ denote the number of edges $e$ in $Q$ with $t(e)=v$.
Before the
\textsc{QueueExtract}, we have $s(v)>0$; it decreases by
one with the extraction. We consider two cases:

Case 1: $q(v)=0$. In this case inequality in the $v$-th coordinate of
(\ref{eq:enougharrows}) continues to hold since the right-hand-side is
zero; all other coordinates in the inequality are unchanged. So
(\ref{eq:enougharrows}) is preserved in this case.

Case 2: $q(v)>0$. In this case \textsc{VisitVertex} adds, for each
$w$, $m_{vw}$ edges $v \to w$ into $Q$, and decrements $q(v)$. Thus the
increase in $s$ and the decrease in the sum on the left-hand side of
(\ref{eq:enougharrows}) exactly cancel out.

Hence both connectedness and having enough arrows are preserved. At
the end of the algorithm, there are no edges left in $Q$;
(\ref{eq:enougharrows}) implies $\textbf{0}=\bs\ge\bq\ge\textbf{0}$,
that is, we have reduced all the quotas to zero, and the algorithm has
succeeded.
\end{proof}

\subsubsection*{Remarks}
We revisit the special case of ordinary graph search of a directed
graph $G$ from a particular vertex $*$. Assume all vertices are
reachable from $*$. We have $q(v)=1$ for all $v\in V(G)$. But, by
connectivity, each vertex in $G$ must either have an edge coming into
it, or must be the start vertex $*$. Thus, in this special case,
having enough arrows is a consequence of connectivity, explaining why
the issue does not become apparent for ordinary graph traversal.

The enough arrows theorem has a very similar flavor to the following
theorem \cite[Theorem 5.6.1]{stanley2001enumerative} characterizing
directed graphs with Eulerian circuits; namely, a global connectivity
condition and a local degree condition. We state it here since
we'll need it in section \ref{sec:mqf}.

\begin{theorem}
\label{eulertheorem}
 A digraph without isolated vertices is Eulerian if
and only if it is connected and balanced (i.e.~
$\textrm{indeg}(v)=\textrm{outdeg}(v)$ for all vertices $v$.)
\end{theorem}

\section{Applications}
\label{sec:applications}
\subsection*{DFA expansion and Myhill-Nerode class sizes}
A \df{deterministic finite-state automaton}, or DFA, is a tuple
$\cD=(S,\Sigma,\d,i,a)$, where $S$ is a finite set of \df{states},
$\Sigma$ is an alphabet, $\d:S\times\Sigma\to S$ is the \df{transition
  map}, $i\in S$ is the \df{initial state}, and $a\subset S$ are the
\df{accept states}. It is useful to think of a DFA as a directed
multigraph over $S$; for each $i\in S$ and $s\in\Sigma$ there is a
directed edge from $i$ to $\d(i,s)$ with label $s$.

The transition map $\d$ has a unique extension to a map
$\d:S\times\Sigma^*\to S$ satisfying
$$\d(s,w_1 w_2)=\d(\d(s,w_1),w_2)$$
for all states $s$ and strings $w_1$, $w_2\in\Sigma^*$.
\footnote{That is, $\d$ defines a semigroup action of $\Sigma^*$ on
  $S$.} ($\d(s,w)$ just starts at $s$ and then applies the unary
operators specified by the symbols in $w$.) The automaton $\cD$
\df{accepts} a string $w$ iff $\d(i,w)\in a$; that is, the path
defined by $w$, starting at the initial state, ends at an accept
state. The automaton is called \df{connected} if the extension
$\d:S\times\Sigma^*\to S$ is onto; that is, all states are reachable
by some path from the initial state.  The set of strings accepted by
$\cD$ is called \df{the language recognized by $\cD$}.  For the
purposes of this paper, a language to be \df{regular} iff it is
recognized by some DFA.

Given a regular language $\cL$, the Myhill-Nerode theorem
\cite[Ch.~3]{hopcroftullman1979} implies that there is a unique
minimal DFA $\cD_\cL$ which recognizes $\cL$. Furthermore, if $\cD$ is
any connected DFA recognizing $\cL$, then there is a quotient map
$\phi:\cD\to\cD_\cL$ which is a homomorphism in the sense of universal
algebra \cite{burris1981course}. That is, $\phi$ maps each state of
$\cD$ to a state of $\cD_\cL$, such that transitions are preserved:
\begin{equation}
\label{eq:homomorphism}
\d(\phi(v),s)=\phi(\d(v,s)) \textrm{\ for $v\in\cD$, $s\in\Sigma$}
\end{equation}
Not surprisingly, $\cD_\cL$ is connected (for if it had unreachable
states, those could be omitted to yield a smaller automaton
recognizing $\cL$.)

As in \cite{fink2017streaming}, we might want to be able to construct
and count larger DFAS recognizing $\cL$. We can use the enough arrows theorem to effectively characterize the
possible sizes of the Myhill-Nerode equivalence classes in a connected
DFA recognizing a language $\cL$.  If $\cD$ is connected, all of its
states are reachable by a graph search from the initial state of
$\cD$.  The Myhill-Nerode theorem implies that the corresponding graph
search tree in $\cD$ corresponds to a quota search in $\cD_\cL$, with
the quota for each state in $\cD_\cL$ being the Myhill-Nerode
equivalence class size. Therefore, the graph of $\cD_\cL$, with these
quotas and the start portfolio consisting of just the initial state,
must satisfy (\ref{eq:enougharrows}).

Furthermore, the converse direction of the theorem implies that
\emph{any} collection of class sizes having enough arrows is
achievable, since the connectedness of the minimal DFA $\cD_\cL$ is
automatic. The quota search tree that witnesses the connectivity of
$\cD$ represents a construction of part of the transition map $\d$ for
$\cD$, but there will be transitions that need assigning. The
remaining transitions can be assigned completely arbitrarily, subject
to the homomorphism constraint (\ref{eq:homomorphism}).  This not only
characterizes the sizes of the Myhill-Nerode classes that can arise in
a DFA recognizing $\cL$, it yields an efficient algorithm for
constructing all DFAs realizing those sizes, when the ``enough arrows''
condition holds. We refer to this process as \df{quota-based DFA
  expansion.}

We emphasize that satisfying the connectivity and enough arrows
conditions does \emph{not} guarantee connectivity of a given extension
structure. In particular, it
is not true that if $\cD$ is a connected DFA, and $\cD'\to\cD$ is a
quotient map with preimage sizes satisfying (\ref{eq:enougharrows}),
then $\cD'$ is connected. But the existence of some connected $\cD'$ is
guaranteed.

\subsubsection*{Example: the Fibonacci language}
At the top of Figure \ref{fig:FibonacciDFAExample} is the minimal DFA
$\cD_\cL$ recognizing the ``Fibonacci language'' $\cL$ of strings over
$\Sigma=\{a,b\}$ without two consecutive $b$'s. We expand this DFA to
obtain one with Myhill-Nerode class sizes $3$, $2$ and $3$
respectively, which satisfies the ``enough arrows'' condition
(\ref{eq:enougharrows}) since
$$(1\ 0\ 0) + (3\ 2\ 3)
\begin{pmatrix}1 & 1 & 0\\ 1 & 0 & 1 \\ 0 & 0 & 2\end{pmatrix}
= (6\ 3\ 8)\ge (3\ 2\ 3).$$ Select an initial node for $\cD$ which
maps down to the initial node of $\cD_\cL$, and do a quota search; the
red arrows in the lower diagram in Figure
\ref{fig:FibonacciDFAExample} show the results of a (breadth-first)
quota search. This leaves some remaining transitions which can be
filled in arbitrarily, as long as they map to the correct class. One
set of choices for these arrows is shown in green.

The enough arrows theorem allows us to precisely characterize the
possible class size vectors $(x,y,z)$. (\ref{eq:enougharrows})
requires
$$(1+x+y,x,y+2z)\ge(x,y,z)$$ coordinatewise; the first and last of
these are vacuous
(in general, nodes with self-loops give a vacuous constraint). So the necessary and
sufficient condition for $(x,y,z)$ to be the Myhill-Nerode class sizes
of an automaton recognizing $\cL$ is simply that $x\ge y\ge1$.

\begin{figure}
\caption{Expanding the Fibonacci DFA to a larger connected DFA via
quota search. (a) The original DFA, with quotas $(3,2,3)$; (b) the
expanded DFA. The red edges form a quota tree, guaranteeing
connectivity; the green edges are a random completion to a DFA.}
\label{fig:FibonacciDFAExample}
\begin{tabular}{ccc}
\raisebox{2em}[0pt][0pt]{(a)}& 
\begin{tikzpicture}
[auto, scale=2]
\node at (0,0) [dot=blue] (1) {};  \node [below] at (1.south) {$3$};
\node at (0,0) [ddot=blue] {};
\node at (1,0) [dot=blue] (2) {};  \node [below] at (2.south) {$2$}; 
\node at (1,0) [ddot=blue] {};
\node at (2,0) [dot=blue] (3) {};  \node [below] at (3.south) {$3$};

\node [left of=1] {} edge [->] (1);

\path[->]
 (1) 
    edge [loop above] node [above] {a} (1)
    edge [bend left] node {b} (2)
 (2)
    edge [bend left] node {a} (1)
    edge node {b} (3)
 (3)
    edge [loop right] node {a,b} ();
\end{tikzpicture}
\\
\raisebox{4em}[0pt][0pt]{(b)}& 
\begin{tikzpicture}
[auto, scale=2]
\definecolor{darkgreen}{rgb}{0,0.5,0}
\node at (0,1) [dot] (1a) {};    \node at (0,1) [ddot] {};
\node at (0,0.5) [dot] (1b) {};  \node at (0,0.5) [ddot] {};
\node at (0,0) [dot] (1c) {};    \node at (0,0) [ddot] {};
\node at (1,0.75) [dot] (2a) {}; \node at (1,0.75) [ddot] {};
\node at (1,0.25) [dot] (2b) {}; \node at (1,0.25) [ddot] {};
\node at (2,1) [dot] (3a) {};
\node at (2,0.5) [dot] (3b) {};
\node at (2,0) [dot] (3c) {};

\node [left of=1a] {} edge [->] (1a);

\path[->]
 (1a) 
    edge [red,bend right] node [left] {a} (1b)
    edge [red] node {b} (2a)
 (1b)
    edge [darkgreen] node [right] {a} (1a)
    edge [red] node [near start] {b} (2b)
 (1c)
    edge [darkgreen] node [left] {a} (1b)
    edge [darkgreen] node [near end] {b} (2b)
 (2a)
    edge [red] node [right,near start] {a} (1c)
    edge [red] node {b} (3a)
 (2b)
    edge [darkgreen, bend left] node {a} (1c)
    edge [darkgreen] node {b} (3b)
 (3a)
    edge [red] node {a} (3b)
    edge [red, bend right] node [left, pos=0.1] {b} (3c)
 (3b)
    edge [darkgreen] node {a,b} (3c)
 (3c)
    edge [darkgreen, loop right] node {a,b} ();

\node[draw=blue,dashed,inner sep=3pt, shape=ellipse, fit=(1a) (1b) (1c)] {};
\node[draw=blue,dashed,inner sep=3pt, shape=ellipse, fit=(2a) (2b)     ] {};
\node[draw=blue,dashed,inner sep=3pt, shape=ellipse, fit=(3a) (3b) (3c)] {};
\end{tikzpicture}
\end{tabular}
\end{figure}
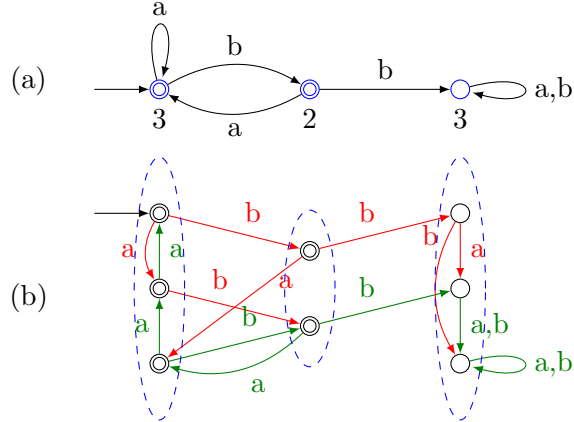

Quota-based DFA construction achieves the goal of effectively
generating random connected DFAs recognizing $\cL$, with specified
Myhill-Nerode equivalence class sizes, in such a way that any
connected DFA can in principle be produced.  The Myhill-Nerode theorem
guarantees that any connected DFA $\cD$ recognizing $\cL$ has a
quotient map down to the (connected) minimal DFA $\cD_\cL$.  The
connectivity of $\cD$ is witnessed by some search tree $\cT$ in the
universal path cover of $\cD$.  If we randomize \textsc{QueueExtract}
so that it returns a randomly selected element of $Q$, we guarantee
that quota search can return $\cT$ as the search forest.  (At this
point, connectivity of $\cD$ implies that the Myhill-Nerode class
sizes must satisfy the ``enough arrows'' condition.) By assigning the
remaining transitions randomly, we guarantee that quota-based DFA
expansion can produce $\cD$. This proves the following theorem:

\begin{theorem}[universality of quota-based DFA expansion]
Let $\cL$ be a regular language, and let $\cD$ be a connected DFA
recognizing $\cL$.  Then:
\begin{itemize}
\item $\cD$ can be constructed from the minimal DFA $\cD_\cL$ by
  quota-based DFA expansion;
\item the Myhill-Nerode equivalence class sizes of $\cD$ must satisfy
  the ``enough arrows'' condition, with the start portfolio being one
  copy of the initial state of $\cD$.
\end{itemize}
\end{theorem}

We remark that even when $\cT$ is chosen uniformly from the set of
quota trees achieving the given M-N class sizes, the resulting DFA
is not sampled uniformly from the connected DFAs with these class
sizes, as different DFAs admit different numbers of spanning trees.
In principle this method could be combined with standard methods
such as Markov Chain Monte Carlo. Efficient uniform DFA generation
is a topic for further research.

\subsection*{$k$ shortest paths}
It turns out that quota search very naturally solves the problem of
constructing a tree which contains the $k$ shortest (or lightest,
if edges are weighted) paths from a source node (or portfolio)
to vertices in a graph $G$.%
\footnote{See \cite{eppstein1998finding} for efficient algorithms and
numerous references. Numerous clever methods have been developed
in connection with this problem; these are no doubt also applicable
in the more general context of quota search.}
For example, when the edge weights are nonnegative, a solution is
to use Dijkstra quota search (DQS) with all quotas initialized to $k$.

For simplicity we assume that the source portfolio is a single vertex
$\ast$, so we're building a tree $T$. Viewed as operating in the
universal path cover $\tilde G$, for each encounter with a vertex
$v=f(t(e))$, DQS keeps track of the distance from $\tilde\ast$ to a
corresponding lift $\tilde v$ in $\tilde G$. The edge $\tilde e$
de-queued at each step extends a path $\tilde p$ in $T$ to a path
$\tilde p\tilde e$ which minimizes
this distance; $\tilde e$ is added to $T$ if $q(v)>0$.  The point now is
that if we have a path to $v$ which is among the $k$ lightest, then
we may assume all initial subpaths are among the lightest $k$ paths
to their corresponding endpoints, and are in $T$ by construction. Thus,
by setting the quota at every vertex to $k$, we are guaranteed that the
quota tree consists of a set of $k$ lightest paths to all vertices.

DQS also solves the network configuration problem in section
\ref{sec:examples}, although since we are minimizing the number of
edges in paths rather than their weighted lengths, breadth-first
quota search gives a simpler solution. As remarked earlier, the
coupon problem described in section \ref{sec:examples} is an
example of the minimum quota arborescence problem; its solution
requires an analogue of Edmonds' algorithm \cite{edmonds1967optimum},
which we will discuss in section \ref{sec:mqf}.

\section{Counting quota trees}
\label{sec:enumeration}

The enumeration of spanning trees is well understood. The most
fundamental result, the matrix-tree theorem, expresses the number of
(directed) spanning trees of a graph $G$ as a principal minor of the
Laplacian of $G$. As a special case, one obtains Cayley's classical
formula that the complete graph $K_n$ has $n^{n-2}$ spanning trees
with a specified root. These turn out to be special cases of a more
general result for quota trees.

As usual, $G$ is a directed (multi)graph, possibly with loops,
having $m_{ij}$ distinct
edges from vertex $i$ to vertex $j$. Let $q:V(G)\to\Z$ be a quota
function, $s:V(G)\to\Z$ a start portfolio, and $M=(m_{ij})$ the
adjacency matrix of $G$. The following symbol is indispensable in
expressing counts of quota trees.

\newcommand{\ba}{\mathbf{a}}
\renewcommand{\bb}{\mathbf{b}}

Given a directed multigraph $G$ with $n\times n$ adjacency matrix
$M=(m_{ij})$, and $n$-long vectors $\ba=(a_i)$ and $\bb=(b_i)$, define
the \df{quota symbol}
\begin{equation}
\label{qsymbol}
\qsymbol{\ba}{\bb}{G} :=
\det M(\ba,\bb)\prod_i \binom{a_i}{b_i}(a_i)^{-1},
\end{equation}
where the binomial coefficient $\binom nk$ is zero unless $0\le k\le n$,
the matrix $M(\ba,\bb) = \diag(\ba)-M \diag(\bb)$,
and for any index $i$ with $a_i=0$ we omit the factor
$a_i^{-1}$ and delete the corresponding row and column of $M(\ba,\bb)$.
(We remark that loops in $G$ do not affect $M(\ba,\bb)$ but do affect
the binomial coefficients.)

\begin{theorem}[counting quota forests]
\label{countingforests}
Let $G$, $q$ and $s$ be as above. As in the enough arrows condition
(\ref{eq:enougharrows}), set
$\In_j=\sum_i q_i m_{ij} = \bq M$ where $M=(m_{ij})$ is the adjacency
matrix of $G$. Then the number of quota forests with
quota $\bq$ and start portfolio \textbf{exactly} $\bs$ is given by
$$\qsymbol{\In}{\bq-\bs}{G}.$$
\end{theorem}

The determinant arising in this theorem has a natural combinatorial
interpretation, which we will need. It represents the (weighted)
counts of spanning forests of the subgraph of $G$ determined by the
support of $q$, relative to the start portfolio $s$. In particular,
the determinant is nonzero precisely when the triple $(G,q,s)$
is connected.
To state this precisely, given weights on the edges and vertices of a graph,
define the weight of a tree to be the weight
of its root times the product of the weights of the edges it contains,
and the weight of a forest to be the product of the weights of its
component trees.

\begin{theorem}[matrix interpretation]
\label{matrixinterpretation}\ \\
$\det M(\In,\bq-\bs)$ is the sum of the
weights of all spanning forests of $G$, where vertex $i$ has weight $s_i$,
and an edge $i\to j$ has weight $q_j-s_j$.
\end{theorem}

This result follows immediately from the following ``matrix-forest theorem,''
which is equivalent to (but much more symmetric than) the usual
``matrix-tree theorem'' \cite[Thm.~5.6.4]{stanley2001enumerative}:

\begin{theorem}[matrix-forest]
\label{matrixforest}
Let $G$ be a directed multigraph with adjacency matrix $M=(m_{ij})$. Define
the Laplacian $\Delta G$ to be $\diag(\In)-M$, where
$\In_j=\sum_im_{ij}$. Then for indeterminates $\bs=(s_i)$,
$\det(\diag(\bs)+\Delta G)$ is the sum of the weights of all spanning forests
of $G$, where the weight of an edge $i\to j$ is $m_{ij}$ and the weight of a
vertex is $s_i$.
\end{theorem}

In particular, for any subset $I$ of the vertices, let $s_I$ denote the
monomial $\prod_{i\in I}s_i$. Then the coefficient
$[s_I]\det(\diag(\bs)+\Delta G)$
is the sum of the (edge) weights
of all spanning forests of $G$ with root set equal to $I$.

\begin{corollary}[enough arrows]
A triple $(G,q,s)$ admits an exact quota forest if and only if $q(v)\ge s(v)$
for each $v$, $(G,q,s)$ is connected, and the enough arrows condition
holds at each vertex.
\end{corollary}

\begin{proof}
The $i$-th binomial coefficient in Theorem \ref{countingforests} is
nonzero precisely when the local ``enough arrows'' condition holds at
the $i$-th vertex and $q_i\ge s_i$.  By Theorem \ref{matrixforest},
the determinant in
Theorem \ref{countingforests} is nonzero precisely when there exists at
least one spanning forest of (the support of $q$ in) $G$ whose roots
are contained in the
support of $s$; that is, when $(G,q,s)$ is connected. The enough arrows
theorem now follows immediately from Theorem \ref{countingforests}.
\end{proof}

We remark that the quota symbol simultaneously generalizes binomial
coefficients and spanning trees. When the graph $G$ has one vertex
and no edges, the quota symbol is a single binomial coefficient.
On the other hand, for general $G$, when the quotas are all
$1$, the quota symbol $\qsymbol{\In}{\bq-\bs}{G}$ counts spanning
forests (or trees) of $G$. So it is not surprising that the symbol
satisfies a recurrence which reduces in the former case to Pascal's
rule, and in the latter case to the recurrence for counting spanning
trees by deleting and contracting an edge:
$$\#T(G)=\#T(G\setminus e) + \#T(G/e).$$
While we won't need it here, we state the recurrence for completeness;
its proof is implicit in the proof of Theorem \ref{countingextensions}.

\begin{theorem}[quota symbol recurrence]
\label{qsymbolrecurrence}
The quota symbol (\ref{qsymbol}) can be computed recursively as
follows:
$$\qsymbol{\ba}{\bb}{G}=\left\{
\begin{array}{ll}
  0,&\textrm{unless $\mathbf{0}\le \bb\le\ba$ and $\ba \ge \bb M$,
     in which case:}\\
  1,&\textrm{if $\bb=\mathbf{0}$; else}\\
  0,&\textrm{if $\ba=\bb M$;}
\end{array}
\right.
$$
otherwise
$$\qsymbol{\ba}{\bb}{G}=
 \qsymbol{\ba-\delta_i}{\bb-\delta_i}{G}+\qsymbol{\ba-\delta_i}{\bb}{G}$$
where $\delta_i$ is the vector with a $1$ in the $i$-th position and $0$
elsewhere, and $i$ is an index such that $a_i > (\bb M)_i$.
\end{theorem}

Corresponding to the two variants of quota search we have described,
one might also ask for the number of at-most quota forests. (Of course
the answers agree for trees, but sometimes one or the other expression
is easier to evaluate.)

\begin{corollary}[counting at most quota forests]
\label{countingatmostforests}
Fix $G$, $\bq$ and $\bs$ as in Theorem \ref{countingforests}.
The number of $(G,\bq,\bs')$-quota trees with a start
portfolio $\bs'$ satisfying $\bs'\le\bs$ coordinatewise is given by
$$\qsymbol{\In+\bs}{\bq}{G}.$$
\end{corollary}

Before proving these results, we pause to consider some special cases.
We will let $Q_=(G,q,s)$ denote the number of quota forests of $G$ with
quota $q$ and start portfolio exactly $s$; similarly $Q_\le(G,q,s)$
will count quota forests with start portfolio at most $s$.

\subsubsection*{Example: one-vertex graphs}
Let $R_k$ denote the $k$-leaf rose, having a single vertex and $k$
loops; in this case, quotas and start portfolios are just scalars $q$
and $s$. By Theorem \ref{countingforests} and Corollary
\ref{countingatmostforests}, we find

$$
Q_=(R_k,q,s)=\left\{
\begin{array}{ll}
[q=s]&\textrm{if\ }kq=0;\\
\frac{s}{q}\binom{kq}{q-s}&\textrm{otherwise;}
\end{array}
\right.
$$
$$
Q_\le(R_k,q,s)=\left\{
\begin{array}{ll}
[q=0]&\textrm{if\ }kq+s=0;\\
\frac{s}{kq+s}\binom{kq+s}{q}&\textrm{otherwise.}
\end{array}
\right.
$$
\vspace{1ex}
It's useful to check the
at-most counts. For $k=0$, $Q_\le(R_0,q,s)=\binom sq$ as expected,
since we just select which of the $s$ possible starts get chosen.
$Q_\le(R_1,q,s)=\frac{s}{q+s}\binom{q+s}{q}=\binom{q+s-1}{q}$, as it
counts the number of $s$-tuples $(T_1,\ldots,T_s)$ where each $T_i$ is
a (possibly empty) directed path graph and the total number of nodes is $k$.
For $k=2$ each $T_i$ is now a binary tree, and 
$Q_\le(R_2,q,s)=\frac{s}{2q+s}\binom{2q+s}{q}$ is equal to the entry
$C(q+s-1,q)$ in the so-called Catalan triangle \cite[A009766]{oeis};
when $s=1$ this just counts binary trees on $n$ nodes: $1,1,2,5,14,\ldots$.
(For higher $k$ we get $k$-th level Catalan numbers; see \cite[A069269]{oeis}.)

We remark that the ordinary $q$-generating function for $Q_\le(R_k,q,s)$ is
a hypergeometric series:
\begin{eqnarray*}
  \sum_q Q_\le(R_k,q,s)z^q
  &=& 1 + \sum_{q\ge1} \frac{s}{kq+s}\binom{kq+s}{q}z^q\\
  &=& {}_kF_{k-1}\left(
    \begin{array}{c}
      \frac{s}{k},\frac{s+1}{k},\ldots,\frac{s+k-1}{k}\\
      \frac{s+1}{k-1},\frac{s+2}{k-1},\ldots,\frac{s+k-1}{k-1}
    \end{array}\bigg|
    \frac{k^k}{(k-1)^{k-1}}z\right).
\end{eqnarray*}

We also note the following relationship, which falls in the category of
``combinatorial reciprocity laws'' as described by Stanley
\cite{stanley1974combinatorial}.
When we formally substitute $-s$ in the expression for $Q_\le(R_k,q,s)$,
we obtain
$\frac{-s}{kq-s}\binom{kq-s}{q}$. When $kq\le s$, it turns out that
this counts (up to an alternating sign) the number of ways to select
a set of $q$ disjoint copies of the path graph $P_k$ in the cycle
graph $C_s$. When $k=1$, this reduces to the usual binomial coefficient
reciprocity law, namely that $\binom{-s}{q}$ and $\binom{s}{q}$ count
selections of $q$ objects from $s$ objects respectively with and without
replacement. For general $k$, this
gives a combinatorial reciprocity law for higher-order Catalan triangles.

\subsubsection*{Example: quota trees over $K_n$}
It is natural to consider a start portfolio
consisting of a single vertex, as this situation arises in the context
of spanning trees as well as deterministic automata.  We view the
complete graph $G=K_n$ as a
directed graph with adjacency matrix $J_n-I_n$ (where $J_n$ is as
usual the matrix of all $1$'s.)  We remark that quota trees over $K_n$
can be viewed as trees colored with $n$ colors, having $q_i$ nodes of
color $i$, such that a node cannot share a color with either its
parent or any of its siblings. In the special case of a constant quota
$q$ at each vertex, we get an especially nice answer:
the number of quota trees over $K_n$, with a given start vertex and
constant quota $q$ at each vertex, is
$$\binom{(n-1)q}{q}^n
  \frac{n^{n-2}}{(n-1)^{n-1}((n-2)q+1)}.$$
Taking $q=1$ yields $n^{n-2}$,
so we recover as a special case Cayley's formula for the number of
spanning trees of $K_n$ with a specified root.

\subsubsection*{Example: quota trees over $K_n$ with loops}
Loops don't enter into spanning trees, but are relevant to quota
forests. We remark that loops do not affect the determinant in the
definition of the quota symbol (\ref{qsymbol}), but they do affect the
rest of the terms. As an example, let $K_n^\circ$ be the graph $K_n$
with a loop added at each vertex, so that the adjacency matrix
is the all-ones matrix. Its quota trees correspond to
tree-colorings as in the preceding example, except that a node
is now allowed to share a color with its parent. When the quota is
a constant $q$ at each root, the number of quota trees starting at
any fixed root works out to be
$$\binom{nq}{q}^n\frac1{n(q(n-1)+1)}.$$
For $n=2$, the number of quota trees with quotas
$(i,j)$ and start portfolio at most $(1,0)$ is given by

$\displaystyle
 \left\{\begin{matrix}i+j+1&i+j\\ i&j\end{matrix}\right\}_{K_2^\circ}:$
\qquad
$\begin{array}{c|cccccc}
i\diagdown j & 0 & 1 & 2 & 3 & 4 & 5 \\
\hline
0 & 1 & 0 & 0 & 0 & 0 & 0 \\
1 & 1 & 1 & 1 & 1 & 1 & 1 \\
2 & 1 & 3 & 6 & 10 & 15 & 21 \\
3 & 1 & 6 & 20 & 50 & 105 & 196 \\
4 & 1 & 10 & 50 & 175 & 490 & 1176 \\
5 & 1 & 15 & 105 & 490 & 1764 & 5292 \\
\end{array}
$

\vspace{1ex}
Up to indexing, these are the Narayana numbers (\cite[A001263]{oeis},
\cite[ex.~6.36]{stanley2001enumerative});
they appear in numerous contexts (e.g.~Dyck paths counted by length and peak,
antichains in the poset $2*(k-1)*(n-k)$, the $h$-vector of the dual
simplicial complex to the associahedron $A_n$, etc.)

Notice that the diagonals in the preceding table add up to the Catalan
numbers; this is a special case of a very general fact.
Let $\pi:\tilde G\to G$ be a (not necessarily universal) out-covering,
$q$ and $\tilde q$ quotas on $G$ and $\tilde G$ respectively, and
$s$ and $\tilde s$ start portfolios such that
$s(v)=\sum_{\tilde v\in \pi^{-1}(v)}\tilde s(\tilde v).$
By the discussion in section \ref{sec:quotatrees}, given a $(G,q,s)$
quota forest $F$, once we lift the root of each tree to an arbitrary preimage
in $\tilde G$, this determines a unique lift of $F$. Thus,
counting quota trees in $\tilde G$ refines the counting
of quota forests in $G$ in the sense that
$$Q_=(G,q,s)=\sum_{\tilde q}Q_=(\tilde G,\tilde q,\tilde s),$$
where the sum ranges over all (achievable) quotas $\tilde q$ such that
$$\sum_{\tilde v\in \pi^{-1}(v)}\tilde q(\tilde v)=q(v)$$
for all $v\in V(G)$.

Returning to the current example, since $K_2^\circ$ has constant outdegree
2, one can construct an out-covering $K_2^\circ\to R_2$.  So the number
of quota trees in $K_2^\circ$ where the quota has $l_1$-norm $n$ is the
number of quota trees in $R_2$ with quota $n$, which we have already
seen is given by a Catalan number.

More generally, there are five essentially different ways to write
down a strongly connected rooted two-vertex graphs with outdegree 2.
In each case, the diagonal sums of the quota tree counts are Catalan
numbers, but the quotas reflect different interesting properties of
the binary trees.  All five cases appear as different entries in Sloane
\cite{oeis}; we list these as the first five rows of Table
\ref{sloanetable}, which collects a number of two-vertex graphs whose
quota tree counts have already been studied in other contexts.

\def\bbox{\path [use as bounding box] (-0.5,-0.5) rectangle (1.8, 0.5);}

\begin{table}
\label{sloanetable}
$$\begin{array}{c|l}
G&\textrm{Corresponding entry in Sloane \cite{oeis}}\\
\hline
\lower0.4cm\hbox{\begin{tikzpicture}
\bbox
\node [ddot,fill] (A) {};  \node [ddot, right of=A] (B) {};
\draw[->] (A) to [bend left] (B);
\draw[->] (A) to [loop left] ();
\draw[->] (B) to [loop right] ();
\draw[->] (B) to [bend left] (A);
\end{tikzpicture}}
 & \parbox[c]{3.8in}{A001263, Narayana numbers}
\\
\lower0.4cm\hbox{\begin{tikzpicture}
\bbox
\node [ddot,fill] (A) {};  \node [ddot, right of=A] (B) {};
\draw[->] (A) -- (B);
\draw[->] (A) to [bend left] (B);
\draw[->] (B) to [bend left] (A);
\draw[->] (B) to [loop right] ();
\end{tikzpicture}}
 & \parbox[c]{3.8in}{A127157, ordered trees with $n$ edges
      and $2k$ nodes of odd degree}
\\
\lower0.4cm\hbox{\begin{tikzpicture}
\bbox
\node [ddot,fill] (A) {};  \node [ddot, right of=A] (B) {};
\draw[->] (A) -- (B);
\draw[->] (A) to [loop left] ();
\draw[->] (B) to [in=345, out=15, loop] ();
\draw[->] (B) to [in=330, out=30, controls= +(30:1) and +(330:1)] (B);
\end{tikzpicture}}
 & \parbox[c]{3.8in}{A9766, the Catalan triangle again}
\\
\lower0.4cm\hbox{\begin{tikzpicture}
\bbox
\node [ddot,fill] (A) {};  \node [ddot, right of=A] (B) {};
\draw[->] (A) to [bend left] (B);
\draw[->] (A) to [loop left] ();
\draw[->] (B) -- (A);
\draw[->] (B) to [bend left] (A);
\end{tikzpicture}}
 & \parbox[c]{3.8in}{A108759, ordered trees with $n$ edges
      containing $k$ (non-root) nodes adjacent to a leaf}
\\
\lower0.4cm\hbox{\begin{tikzpicture}
\bbox
\node [ddot,fill] (A) {};  \node [ddot, right of=A] (B) {};
\draw[->] (A) to [bend left=45] (B);
\draw[->] (A) to [bend left=15] (B);
\draw[->] (B) to [bend left=45] (A);
\draw[->] (B) to [bend left=15] (A);
\end{tikzpicture}}
 & \parbox[c]{3.8in}{A212206, ``pat'' permutations
      of $[n]$ with $k$ descents}
\\
\lower0.4cm\hbox{\begin{tikzpicture}
\bbox
\node [ddot,fill] (A) {};  \node [ddot, right of=A] (B) {};
\draw[->] (A) to [bend left] (B);
\draw[->] (B) to [bend left] (A);
\draw[->] (B) to [loop right] ();
\end{tikzpicture}}
 & \parbox[c]{3.8in}{A055151, Motzkin polynomial coefficients; diagonal sums
 are Motzkin numbers A001006}
\\

\lower0.4cm\hbox{\begin{tikzpicture}
\bbox
\node [ddot,fill] (A) {};  \node [ddot, right of=A] (B) {};
\draw[->] (A) to [loop left] ();
\draw[->] (A) to [bend left] (B);
\draw[->] (A) -- (B);
\draw[->] (B) to [bend left] (A);
\draw[->] (B) to [in=345, out=15, loop] ();
\draw[->] (B) to [in=330, out=30, controls= +(30:1) and +(330:1)] (B);
\end{tikzpicture}}
 & \parbox[c]{3.8in}{A108767, 2-Dyck paths of order $n$ with $k$ peaks}
\\

\lower0.4cm\hbox{\begin{tikzpicture}
\bbox
\node [ddot,fill] (A) {};  \node [ddot, right of=A] (B) {};
\draw[->] (A) to [loop left] ();
\draw[->] (A) to [bend left=45] (B);
\draw[->] (A) to [bend left=15] (B);
\draw[->] (B) to [bend left=45] (A);
\draw[->] (B) to [bend left=15] (A);
\draw[->] (B) to [loop right] ();
\end{tikzpicture}}
 & \parbox[c]{3.8in}{A278881}
\\

\lower0.4cm\hbox{\begin{tikzpicture}
\bbox
\node [ddot,fill] (A) {};  \node [ddot,fill, right of=A] (B) {};
\draw[->] (A) to [bend left] (B);
\draw[->] (A) to [loop left] ();
\draw[->] (B) to [loop right] ();
\draw[->] (B) to [bend left] (A);
\end{tikzpicture}}
 & \parbox[c]{3.8in}{A145596, generalized Narayana numbers;
 A214457, rhombic tilings of an $(n,k,1,1,n,k,1,1)$ octagon}
\\

\lower0.4cm\hbox{\begin{tikzpicture}
\bbox
\node [ddot,fill] (A) {};  \node [ddot, right of=A] (B) {};
\draw[->] (A) to [loop left] ();
\draw[->] (A) to [bend left=22.5] (B);
\draw[->] (A) to [bend left=45] (B);
\draw[->] (A) -- (B);
\draw[->] (B) to [bend left] (A);
\draw[->] (B) to [in=345, out=15, loop] ();
\draw[->] (B) to [in=330, out=30, controls= +(30:1) and +(330:1)] (B);
\draw[->] (B) to [in=315, out=45, controls= +(45:1.5) and +(315:1.5)] (B);
\end{tikzpicture}}
 & \parbox[c]{3.8in}{A173020, 3-Runyon numbers}
\\

\lower0.4cm\hbox{\begin{tikzpicture}
\bbox
\node [ddot,fill] (A) {};  \node [ddot, right of=A] (B) {};
\draw[->] (A) -- (B);
\draw[->] (A) to [bend left] (B);
\draw[->] (B) to [bend left] (A);
\end{tikzpicture}}
& \parbox[c]{3.8in}{A068763, related to generalized Catalan sequences}
\\
\hline
\end{array}
$$
\caption{Some two-vertex graphs whose quota tree counts appear, possibly
re-indexed, in Sloane's Encyclopedia of Integer Sequences. In each case,
the start portfolio is one copy of each filled-in vertex.}
\end{table}

\subsubsection*{Example: quota forests over $K_n$ with symmetric roots}
It is even more symmetrical to count quota forests over $K_n$,
where we take both $q$ and $s$ to be constant over all vertices.
The quota tree count is
$$\binom{(n-1)q}{q-s}^n\frac{(nq-s)^{n-1}s}{(n-1)^{n-1}q^n}.$$
In particular, if $q=s$, the count is exactly one, reflecting the
fact that each tree in the forest is an isolated node.

\subsubsection*{Example: path graphs}
The path graph $P_n$ has only a single spanning tree from any root;
however, quota trees are much more interesting.  Intuitively, we have
$n$ parallel semitransparent panes of glass; at each one, a laser beam
can pass through, reflect, both, or neither. When we fire a beam
into one pane, the trajectory is then a tree immersing into
$P_n$, whose quotas count the number of times each pane is encountered.
If all quotas are $q$, and the beam is initially fired into one of the
outer panes, the number of quota trees works out to
$$\left(\frac12\binom{2q}{q}\right)^{n-2}=a_q^{n-2},$$
where $a_q=(1,3,10,35,126,\cdots)$ is sequence A001700 in Sloane.
When we fire the laser into any one of the internal panes, the
answer works out to $c_q a_q^{n-3}$, where $c_q=\binom{2q+1}q/(2q+1)$
is the $q$-th Catalan number.

\subsubsection*{Example: cycle graphs}
With the notation of the preceding example, the cycle graph $C_n$ has 
$$\binom{2q}{q}^n\frac{n}{2^{n-1}(q+1)}=\frac{2n\, a_q^n}{q+1}$$ quota
trees from any fixed root, when all vertex quotas are set to $q$.

\subsubsection*{Proof of Theorem \ref{countingforests}}
The strategy is to write down a functional equation jointly satisfied
by the generating functions for quota trees rooted at all vertices of
$G$, and solve it using the multivariate Lagrange inversion
formula.\footnote{Problem \textbf{3.3.42} in
  \cite{goulden2004combinatorial} is very similar; however, it counts
  trees rather than forests, and omits the immersion condition.}
Following \cite{goulden2004combinatorial}, let $R$ be a ring with
unity, $R[[\bl]]_1$ the set of formal power series in
$\bl=(\lambda_1,\ldots,\lambda_n)$ over $R$ with invertible constant
term, and $R((\bl))$ the ring of formal Laurent series over $R$.

\begin{theorem}[Multivariate Lagrange]%
\cite[Th.~1.2.9]{goulden2004combinatorial}
\ \\
Suppose $\bw=(w_1(\bt),\ldots,w_n(\bt))$ jointly satisfy the functional
equations $w_i(\bt)=t_i\phi_i(\bw)$, where $\bt=(t_1,\ldots,t_n)$.
Let $f(\bl)\in R((\bl))$ and
$\bp=(\phi_1(\bl),\ldots,\phi_n(\bl))$, where $\phi_i\in R[[\bl]]_1$.
Then
$$f(\bw(\bt))=\sum_\bq \bt^\bq[\bl^\bq]\left\{
  f(\bl)\bp^\bq(\bl)\left\Vert\delta_{ij}-\frac{\lambda_j}{\phi_i(\bl)}
  \frac{\partial\phi_i(\bl)}{\partial\lambda_j}\right\Vert\right\}.$$
\end{theorem}

Given a graph $G$ with $n$ vertices, we will take $w_i(\bt)$ to be
the generating function
$$w_i(\bt)=\sum_T \bt^{\bq(T)}
  =\sum_T t_1^{q_1(T)}\cdots t_n^{q_n(T)},$$
where $T$ ranges over all quota trees rooted at vertex $i$, and $q_j(T)$
is the number of occurrences of vertex $j$ in $T$. The first observation
is that the $w_i$'s jointly satisfy the functional equation
\begin{equation}
\label{eq:functionalequation}
w_i(\bt) = t_i\prod_j (1+w_j(\bt))^{m_{ij}}
\end{equation}
where the product ranges over all directed edges $i\to j$ in $G$, since
by the immersion property, a quota tree with root $i$ can have at most
one copy of each of the $m_{ij}$ outgoing edges from $i$ to any vertex
$j$. Thus, in the Lagrange inversion theorem, we will take
\begin{equation}
\phi_i(\bl)=\prod_j (1+\lambda_j)^{m_{ij}}=\prod_j\beta_j^{m_{ij}}
\end{equation}
where $\beta_j$ represents the binomial $1+\lambda_j$.

It is immediate that
$$\frac{\lambda_j}{\phi_i(\bl)}
  \frac{\partial\phi_i(\bl)}{\partial\lambda_j}
  = \frac{m_{ij}\lambda_j}{\beta_j}.$$
We also have that
$$\bp^\bq(\bl)=\prod_i\prod_j\beta_j^{m_{ij}}=\bb^{\bq\cdot M}=\bb^\In$$
where $M$ is the adjacency matrix of $G$, and $\In=\bq\cdot M$ as
in Theorem \ref{countingforests}. Hence

$$ [\bt^\bq] f(\bw(\bt))= [\bl^\bq]\left\{
  f(\bl)\bb^\In
  \left\Vert\delta_{ij}-\frac{m_{ij}\lambda_j}{\beta_j}
  \right\Vert\right\}.$$

At this point we specialize $f$. If we were to set $f(\bl)=\lambda_i$,
we would extract precisely the generating function for quota trees
with root $i$.  Since the generating function for quota forests with
the sum of two portfolios is the product of the individual generating
functions, we obtain the generating function for portfolio $\bs$ by
taking
$$f(\bl)=\lambda_1^{s_1}\cdots\lambda_n^{s_n}=\bl^\bs.$$
Hence
\begin{eqnarray}
 [\bt^\bq] f(\bw(\bt)) &=& [\bl^\bq]\left\{
  \bl^\bs\bb^\In
  \left\Vert\delta_{ij}-\frac{m_{ij}\lambda_j}{\beta_j}
  \right\Vert\right\}\\
  &=& [\bl^{\bq-\bs}]\left\{\bb^\In
  \left\Vert\delta_{ij}-\frac{m_{ij}\lambda_j}{\beta_j}
  \right\Vert\right\}
\end{eqnarray}

We have $[\lambda_j^{k_j}]\beta_j^{n_j}=\binom{n_j}{k_j}$, so we can write
$$ [\bl^{\bq-\bs}]\{\bb^\In\} = \binom{\In}{\bq-\bs}$$
where the right-hand side represents the product of the
individual binomial coefficients. The determinant is, by Laplace
expansion, a sum of products
$$\prod_j l_j\left(\frac{\lambda_j}{\beta_j}\right)$$
where each $l_j$ is a linear function. Together with the observation
that
$$[\lambda_j^{k_j}]\left\{\frac{\lambda_j}{\beta_j}\beta_j^{n_j}\right\}
  = [\lambda_j^{k_j-1}]\beta_j^{n_j-1}=\binom{n_j-1}{k_j-1}
  = \frac{k_j}{n_j} \binom{n_j}{k_j},$$
we arrive at:
\begin{eqnarray}
[\bt^\bq]\{\bw(\bt)^\bs\}
\label{firstmatrixform}
  &=&\binom{\In}{\bq-\bs}
  \left\Vert\delta_{ij}-\frac{m_{ij}(q_j-s_j)}{\In_j}\right\Vert\\
  &=&\binom{\In}{\bq-\bs}\cdot\frac{1}{\In}
  \left\Vert\delta_{ij}\In_j-m_{ij}(q_j-s_j)\right\Vert
\end{eqnarray}
where $1/\In$ is the reciprocal of the product of the $\In_j$'s, yielding
Theorem \ref{countingforests}. (If any $\In_j=0$, the corresponding column
of the matrix in (\ref{firstmatrixform}) has a $1$ in position $j$ and
zeroes elsewhere, so there's no denominator there to clear, and we can
remove that row and column from the matrix.)\qed

\subsubsection*{Proof of Corollary \ref{countingatmostforests}}
We apply Theorem \ref{countingforests}
to an augmentation $\hat G$ of $G$. Given $G$, $q$ and $s$, form a
directed multigraph $\hat G$ by adjoining a new vertex $\ast$, with
$s_i$ edges from $\ast$ to node $i$. We take $\hat q(i)$ to be $q(i)$
if $i$ is a node of $G$, and $\hat q(\ast)=1$. Our start portfolio
$\hat s$ consists of the single vertex $\ast$. A quota tree which is
exact for $(\hat G,\hat q,\hat s)$ is equivalent to a quota forest in
$G$ which starts from each vertex $i$ of $G$ at most $s_i$ times.
Observing that $\hat\In_j=\In_j+s_j$ if $j\in G$, and $\In_\ast=0$,
Theorem \ref{countingforests} implies that the number of ``at most''
quota trees for $(G,q,s)$ is
$$
\binom{\hat\In}{\hat\bq-\hat\bs}
  \left\Vert\delta_{ij}-\frac{\hat m_{ij}(\hat q_j-\hat s_j)}{\hat\In_j}\right\Vert
=
\binom{\In+\bs}{\bq}
  \left\Vert\delta_{ij}-\frac{\hat m_{ij}(\hat q_j-\hat s_j)}{\hat\In_j}\right\Vert
$$
Since $m_{i\ast}=0$ for $i\ne\ast$, the determinant is just the
$(\ast,\ast)$ minor (i.e.~the principal minor corresponding to $G$.)
Removing hats yields the result.
\qed

\subsubsection*{Proof of the matrix-forest theorem}
To deduce the enough arrows theorem (Theorem \ref{thm:maintheorem})
from Theorem \ref{countingforests}, we need to know that the determinant
is nonzero precisely when $(G,q,s)$ is connected. We do this by
interpreting the determinant as a count of spanning forests. The
result needed is equivalent to the usual matrix-tree theorem, but the
formulation we need is not the standard one, so we include a proof
here.

For general directed graphs, a version of the matrix-tree theorem
originally due to Tutte \cite[Thm.~5.6.4]{stanley2001enumerative}
expresses the number of directed spanning trees in a graph as the
determinant of a principal minor of (essentially) the Laplacian of the
graph. The following version of the
matrix-tree theorem, which enumerates rooted spanning trees by weight,
suffices for our purposes.

Let $G$ be a directed graph on $n$ vertices $\{v_1,\ldots,v_n\}$. An
inward spanning tree rooted at $r$ is a directed subgraph of $G$ without
cycles such that every vertex has exactly one outgoing edge, except
the root $r$ which has no outgoing edges. The weight of an edge
$v_i\to v_j$ is $a_{ij}$, the weight of a graph
is the product of its edge weights, and the weight of a set of graphs
is the sum of the weights of the graphs in the set.

\begin{theorem}[Matrix-tree theorem for directed graphs]
\cite[\S4]{zeilberger1985combinatorial}
The weight of the set of inward spanning trees of $G$ rooted at $r$ is
the determinant of the matrix obtained by deleting the $r$-th row and
$r$-th column from the matrix 
$$\begin{pmatrix}
a_{12}+\cdots+a_{1n} & -a_{12} & \cdots & -a_{1n}\\
-a_{21} & a_{21}+a_{23}+\cdots+a_{2n} & \cdots & -a_{2n}\\
\vdots &\vdots & \ddots & \vdots\\
-a_{n1} & -a_{n2} & \cdots & a_{n1}+a_{n2}+\cdots+a_{2,n-1}\\
\end{pmatrix}.$$
\end{theorem}

If we only want to count trees, we can set $a_{ij}$ to be the number
of edges from $v_i$ to $v_j$.  Cayley's formula follows by setting
$a_{ij}=1$ for all $i\ne j$.

To prove Theorem \ref{matrixforest} from the matrix-tree theorem,
first transpose the adjacency matrix since we are interested in
outward trees rather than inward trees.  Construct $\hat G$ as in the
proof of the at-most quota tree formula: $\hat G$ has an additional
vertex $\ast$, and an edge from $\ast$ to each node $i$ in $G$ with
weight $s_i$. A spanning tree in $\hat G$ rooted at $\ast$ corresponds
to a spanning forest $F$ in $G$, where vertex $i$ is a root of a tree
in $F$ iff the spanning tree in $\hat G$ uses the edge $\ast\to
i$. Now apply the matrix-tree theorem to $\hat G$.\qed

(We remark that it is easy to deduce the matrix-tree theorem from the
matrix-forest theorem by setting $s_i$ to be $1$ if $i$ is the desired
root, and $0$ otherwise. So the formulations are equivalent.)

\section{Uniformly generating quota forests}
\label{sec:generation}

The Colbourn-Myrvold-Neufeld algorithm \cite{colbourn1996algorithms}
is an excellent example of one general method of uniformly sampling
from the set of spanning trees (or arborescences, in the case of
directed graphs). One builds up a tree $T$ by
considering each edge $e$ of $G$ in turn, and flipping a biased coin
to decide whether to attach it to $T$.  At
each step of a graph search, the matrix-tree theorem allows us to
efficiently count the number of ways to extend the partial tree $T$ to
a spanning tree, using only edges we have not yet considered, and the
number of these extensions which use the edge $e$. This allows us to
set the coin bias at each step so as to guarantee
uniformly distributed outputs. The determinant of the Laplacian can be
updated easily at each step, as either collapsing or deleting an edge
of $G$ results in a rank-one update to the Laplacian.

We extend this idea to quota trees by running a version of quota
search.  At each stage of this algorithm we have two forests
$\FU\subset\FS$ which immerse into $G$. $\FU$ is the search forest
constructed so far (the roots together with edges we ``used''), while
$\FS$ is the set of all edges we have dequeued from the priority queue
(which contains, in addition to the edges we ``used,'' some other
edges that we ``skipped.'')  Initially $\FU=\FS$ consists simply of the
roots determined by the start portfolio. We will need to know the
number of extensions of $\FU$ to a forest
achieving the given quotas, while avoiding using any skipped edges
(i.e.~those in $\FS\setminus\FU$.)  To this end, we prove the following
extension of theorem \ref{countingforests}.

\begin{theorem}
\label{countingextensions}
With notation as above, let $\seen(v)$ and $\used(v)$
be the number of edges ending at $v$ in $\FS$ and $\FU$
respectively. Then the number of quota forests in $G$ achieving
quotas $\bq$, extending $\FU$ and containing no edges from
$\FS\setminus\FU$, is
$$\qsymbol{\In - \seen}{\bq - \bs - \used}{G}.$$
\end{theorem}

Before proving this theorem, we outline the resulting algorithm for
random quota tree generation.  In each iteration of the main loop in
quota search, we first dequeue an edge, and then either add the edge
to the forest or not.  By
Theorem \ref{countingextensions}, we want to add the edge with
probability
$$\qsymbol{\In - \seen - \delta_v}{\bq - \bs - \used - \delta_v}{G}
  \Big/\qsymbol{\In - \seen}{\bq - \bs - \used}{G},$$
where $\delta_v$ has a $1$ in the $v$ position and $0$ elsewhere.
We then replace $\seen$ by $\seen+\delta_v$, and if we used the edge
we replace $\used$ by $\used+\delta_v$. Either way, updating the
quota symbol requires updating a single binomial coefficient and
a single column in the matrix in (\ref{qsymbol}). As this is a
rank-one update, the determinant or inverse can be updated in
$O(n^2)$ time, where $n=|V(G)|$ (see \cite{golub2012matrix}).
As a result, we can generate a uniformly sampled quota forest in
$$O\left(V^2 \sum_v q(v) Adj(v) (\lg V+\lg q(v))\right)$$
multiprecision operations.

\subsubsection*{Proof of Theorem \ref{countingextensions}}
As was the case for Corollary \ref{countingatmostforests}, we can
proceed by applying Theorem \ref{countingforests} to an
appropriate graph $\hat G$. In this case, $\hat G$ completely
encodes the state of a run of quota search. At any time, $\hat G$
consists of $\FU\cup G$, together with one additional edge
from $\FU$ to $G$ for each edge in the priority queue $Q$.
We will define $\hat q$ and $\hat s$
so that $(\hat G,\hat q,\hat s)$-forests correspond
to $(G,q,s)$-forests which extend $\FU$ and contain no edges from
$\FS\setminus\FU$.

Initially, $\FU$ consists of the roots specified by the start
portfolio. For each $v\in\FU$, and each edge $e\in\outst{f(v)}$,
insert an edge $\hat e$ from $v$ to $t(f(v))$ (which lies in $G$).
Set $\hat s(v)=1$ for each of the roots, and
$\hat s(v)=0$ on $G$. Set $\hat q(v)=1$ for $v\in\FU$, and
$\hat q(v)=q(v)-s(v)$ for $v\in G$.  Clearly $(G,q,s)$-forests
$G$ correspond to $(\hat G,\hat q,\hat s)$-forests; we have just
lifted the start portfolio out of $G$. The start portfolio $\hat s$
is unchanged from this point forward.

We will arrange that, at every step, $(\hat G,\hat q,\hat s)$-forests
correspond precisely to $(G,q,s)$-forests which extend $\FU$ and
which use no edges in $\FS\setminus\FU$.
When we dequeue an edge $e$ in quota search, that edge
corresponds to an edge $v\to w$ in $\hat G$ from $\FU$ to $G$.
Remove the edge from $\hat G$. If quota search decides to add the
edge to the search forest $\FU$, we keep track of this in $\hat G$
by adding an edge from $v$ to a new vertex $\hat w\in\FU$,
extending the immersion $f:\FU\to G$ by mapping $f(\hat w)=w$.
Additionally, quota search inserts some of the edges $e'\in\outst{w}$
into the queue $Q$; for
each such $e'$, we add to $\hat G$ a corresponding edge $\hat{e'}$
from $\hat w$ to $t(e')\in G$.
We assign $\hat q(\hat w)=1$, and decrement $\hat q(w)$.

It remains to evaluate the quota symbol
$$\qsymbol{\hat\In}{\hat\bq-\hat\bs}{G}.$$
If $v$ is a root of a tree in $\FU$ then
$\hat\In(v)=\hat q(v)-\hat s(v)=0$; for other vertices $v\in\FU$
we have $\hat\In(v)=\hat q(v)-s(v)=1$, so the binomial coefficients
from $\FU$ are all $1$. As there are no edges from $G$ to $\FU$ in
$\hat G$, the adjacency matrix of $\hat G$ has a block structure
so the determinant in the quota symbol is the product of the
determinants coming from $\FU$ and $G$ independently. But the former
determinant is one, as it counts the single quota tree in $\FU$.

Thus $\FU$ contributes nothing to the symbol; we are left with the
contribution from $G$.
But for $v\in G$, we have $\hat\In(v)=\In(v)-\seen(v)$,
$\hat q(v)=q(v)-s(v)-\used(v)$, and $\hat s(v)=0$, and the result
now follows from Theorem \ref{countingforests}.\qed

\section{Minimum-Weight Quota Forests}
\label{sec:mqf}

In this section we consider the problem of finding a minimum-weight
quota forest (MQF). Assume we are given an achievable triple $(G,q,s)$
and a weight function $w:E(G)\to\R$. We assume without loss that
$q(v)>0$ for all $v$.  If $f:T\to G$ is an immersion, the pullback
$(f^*w)(e)=w(f(e))$ defines natural edge weights on $T$.  We define
the weight of $f:T\to\R$ to be
$$(f^*w)(T)=\sum_{e\in E(T)}(f^*w)(e).$$ When the immersion $f$ is
understood, we may abuse notation by writing $w$ for $f^*w$.

The minimum-weight quota forest problem is then simply: given an
achievable triple $(G,q,s)$ and weight function $w$, find a quota
forest $T$ which minimizes $w(T)$. (As usual, since edges are
directed, ``tree'' here really means means ``arborescence,'' and
``forest'' really means ``branching.'')  When the quotas are all $1$,
and $||s||_1=1$, this reduces to the minimum spanning arborescence
(MSA) problem. We begin by reviewing Edmonds' MSA algorithm
\cite{edmonds1967optimum}, as it forms the basis of our algorithm for
minimum-weight quota trees.

\subsection*{Edmonds' algorithm}

Edmonds' algorithm (following
\cite[\S3.4]{gondran1984graphs} and \cite{edmonds1967optimum})
takes as input a triple $(G,r,w)$, where $G$ is a
directed loop-free multigraph, $r$ is a specified root, and $w$ is
a weight function on edges. It proceeds as follows:
\begin{enumerate}
\item Form the subgraph $H\subset G$ by taking, for each vertex $v$
  of $G$ other than $r$, the lowest-weight edge into $v$.

\item If $H$ has no circuits, it is a minimum spanning arborescence;
  return $H$.

\item Otherwise, $H$ contains a circuit $C$. Collapse $C$ to form
  $G'=G/C$, and reweight the surviving edges: set
  $w'(e)=w(e)$ unless $e$ is an edge from $G\setminus C$ to $C$, in
  which case 
  set $w'(e)=w(e) - w(e')$, where $e'$ is the edge in $C$ such that
  $t(e)=t(e')$.

\item Recurse to get a MSA $T'$ for $(G',r, w')$.
  $T=T'\cup C\setminus\{e'\}$ is a MSA for $(G,r,w)$, and $w(T)=w(T')+w(C)$.
\end{enumerate}

Edmonds formulated the MSA problem as a linear program.
Define an \df{inventory} to be a set of values $\{x_e\mid e\in G\}$;
we identify any subgraph $H\subset G$ with its inventory
$x_e = [e\in H]$.  The inventory $x=\{x_e\}$ of a spanning arborescence
satisfies the following linear constraints:
\begin{description}
\item[edge] For each edge $e\in E(G)$, $0\le x_e\le 1$;
\item[node] For each node $v\in V(G)$, $\sum_{e\in\inst v}x_e=1$ if
  $v\ne r$, and $0$ if $v=r$;
\item[subset] For each subset $S\subset V(G)$ with $|S|\ge2$,
$$\sum_{e\in E(G[S])}x_e\le |S|-1,$$
where $G[S]$ is the subgraph of $G$ induced by $S$.
\end{description}

Edmonds shows \cite[Theorem 2]{edmonds1967optimum}
that the vertices of the polyhedron defined by these
constraints correspond precisely to arborescences, so in particular
satisfy $x_e\in\{0,1\}$. Thus an MSA corresponds to a vertex
minimizing the weight $\sum_e w(e)x_e$ of the tree.

We remark that step 1 of Edmonds' algorithm ensures that the
\textbf{edge} and \textbf{node} conditions always hold.
The circuit $C$ in step 3 exists precisely when the \textbf{subset}
condition is violated.

\subsection{Extension to minimum-weight quota forests}
We associate to an immersed tree $f:T\to G$ the inventory $x$ given by
$x_e=|f^{-1}(e)|$, so that $w(T)=\sum_e w(e)x_e$. Typically $x_e>1$,
so $T$ cannot be reconstructed uniquely from its edge
inventory. However, if the $x_e$'s are nonnegative integers, it is
useful to associate to $x$ the multigraph having $x_e$ copies of
each edge in $G$; we denote this multigraph by $G[x]$.

Given quotas and a start portfolio, one can write down an analogue of
the linear constraints such that an integer point is feasible if and
only if it is the inventory of a quota forest.  We will develop an
extension of Edmonds' algorithm which produces a minimal-weight
inventory. As in the MSA case, the algorithm guarantees that the
resulting inventory is integral, so it corresponds to at least one
minimum-weight quota forest.  A separate algorithm can be used to
build a forest with that inventory, if the forest itself is needed.

Most of the work is in verifying that the integer solutions to the
linear constraints correspond precisely to inventories of quota forests.

\begin{theorem}[LP characterization of quota forest inventories]
\label{inventoryconstraints}
Let $(G,q,s)$ be given, and let $\{x_e\}$ be the edge inventory of any
quota forest in $G$ achieving quotas $q$ with start portfolio
$s$. Then the following constraints hold:

\begin{description}
\item[edge] For each edge $e\in E(G)$, $0\le x_e\le q(i(e))$;
\item[node] For each node $v\in V(G)$, $\sum_{e\in\inst v}x_e=q(v)-s(v)$;
\item[subset] For each subset $S\subset V(G)$ with $|S|\ge1$,
$$\sum_{e\in E(G[S])}x_e\le \sum_{v\in S}q(v)-1,$$
where $G[S]$ is the subgraph of $G$ induced by $S$.
\end{description}

Conversely, given an achievable triple $(G,q,s)$, and integers
$\{x_e\}$ which satisfy these constraints, there exists a $(G,q,s)$
quota forest with inventory $\{x_e\}$.
\end{theorem}

Note that we now need to check the \textbf{subset} condition on singletons,
since loops can be traversed in MQFs (but do not affect MSAs).

\begin{proof}
The forward direction is similar to the arborescence case. Suppose
$\{x_e\}$ is the inventory of a quota forest. Then \textbf{edge}
follows from the immersion condition, as the number of preimages of $e$
can be at most the number of preimages of the initial vertex $i(e)$ of
$e$. \textbf{node} asserts precisely that the vertex quotas are met.
Finally, by \textbf{node}, we have
$$\sum_{v\in S}\sum_{e\in\inst v}x_e \le \sum_{v\in S}q(v)-s(v).$$
So if \textbf{subset} fails, we must have the equality
$$\sum_{e\in E(G[S])}x_e=\sum_{v\in S}\sum_{e\in\inst v}x_e =
  \sum_{v\in S}q(v)-s(v)=\sum_{v\in S} q(v).$$

In particular, $S$ does not involve any vertices in the start portfolio,
and in the multigraph $G[S][x]$,
both the indegree and outdegree of each vertex $v$ are equal to $q(v)$.
Thus, by the multigraph version of Euler's theorem
(Theorem \ref{eulertheorem}), $G[S][x]$ admits a directed Euler circuit.
That is, there is a map of a directed cycle $C$ onto $G[S]$
hitting each vertex $v$ exactly $q(v)$ times and covering each edge
exactly $x_e$ times. (For later reference, we remark that since $C$ is a
cycle, this map is automatically an immersion.) This means the quotas
for $v\in S$ are entirely used up by edges within $S$, which does not
contain any start vertices. Thus the alleged quota forest contains
no paths from start vertices to $S$, a contradiction. So, the inventory
of a quota forest must satisfy \textbf{subset} as well.

For the reverse direction, we are given an achievable $(G,q,s)$ and an
edge inventory
$\{x_e\}$ satisfying the \textbf{edge}, \textbf{node} and \textbf{subset}
conditions hold; we wish to construct a quota forest with the specified
edge inventory.  If $\mathop{\mathrm{supp}} q=\emptyset$ we are done;
the only feasible inventory is the all-zero inventory,
which is realized by the empty forest.  Thus we may assume at least one
node has positive quota. By achievability, some node $r$
must have $s(r)>0$. We construct $(G, \hat q, \hat s)$ and inventory
$\{\hat x_e\}$ by ``using'' as many edges out of $r$ as possible. Precisely:
\begin{enumerate}
\item Initialize $\hat q=q$, $\hat s=s$, $\hat x_e=x_e$.
\item ``Use'' the root: decrement $\hat q(r)$ and $\hat s(r)$.
\item ``Use'' the edges out of $r$, and allow starts from their
  endpoints. That is, for each edge $e\in\outst r$ with $x_e>0$,
  decrement $\hat x_e$ and increment $\hat s(t(e))$.
\end{enumerate}
It is straightforward to check that $(G,\hat q,\hat s)$ and $\{\hat x_e\}$
satisfy the \textbf{edge}, \textbf{node} and \textbf{subset} conditions,
and $||\hat q||_1=||q||_1-1$, so we are done by induction on the sum of the quotas.
\end{proof}

We remark that both directions of this proof are well-suited for
effective implementation. The subset $S$ in the forward direction is a
start-free source in the strongly connected component graph of the
multigraph $G[x]$, and can thus be found in
$O(V(G)+E(G))$ time. The reverse direction says that we can use a
greedy algorithm to construct a quota forest from a feasible edge
inventory; this takes $O(||q||_1)$ time as stated.

The collapsing step will introduce one further wrinkle.  In analogy
with Edmonds' algorithm, when we find a subset $S$ violating
\textbf{subset}, we will collapse it into a single vertex $v_S$. We
assign $v_S$ a quota of $1$, since WLOG a MQF $T$ will enter $S$
exactly once, and then follow the circuit $C$. However, if $v\in S$
has $q(v)>1$, then for a given edge $e$ from $v$ to $V(G)\setminus S$,
a quota tree $T$ may contain up to $q(v)$ lifts of $e$. To handle this
correctly, we replace the single edge $e$ out of $v$ in $G$ with
$q(v)$ copies of $e$ coming out of $v_S$, all with weight
$w(e)$.

Specifically, we keep track of a ``copy count'' $c_e$ for each
edge $e$ of $G$, typically initialized to $1$. When we collapse $S$,
we multiply $c_e$ by $q(i(e))$ for each edge $e$ connecting $S$ to
$V(G)\setminus S$. In this context,
the \textbf{edge} constraint becomes:
\begin{description}
\item[edge] For each edge $e\in E(G)$, $0\le x_e\le c_e q(i(e))$;
\end{description}

We capture the resulting algorithm for computing a
minimum-weight quota forest as Algorithm \ref{alg:MQF}.
The proof of correctness proceeds essentially identically to that of
Edmonds' algorithm, given Theorem \ref{inventoryconstraints}.

\begin{algorithm}
\caption{Algorithm for computing a minimum-weight quota forest}
\label{alg:MQF}
\noindent\textbf{Input: } an achievable triple $(G,q,s)$ and an
edge-weight function $w$\\
\noindent\textbf{Output: } a minimum-weight quota forest for $(G,q,s,w)$
\begin{enumerate}
\item Form the inventory $x=\{x_e\}$ by taking, for each vertex $v$,
  the lowest-weight $q(v)-s(v)$ possible edges into $v$. (As with
  Edmonds' algorithm, there may be many possibilities here if edge
  weights are not unique.) This will ensure that the \textbf{edge} and
  \textbf{node} conditions hold.  (Since $(G,q,s)$ is achievable, the
  enough arrows condition holds, which guarantees that $v$ has enough
  incoming edges to make this step possible.)

\item Compute the strong-component graph of $G[x]$. A subset $S$ of
  nodes violates \textbf{subset} iff it corresponds to a start-free
  source in the component graph. If no subset $S\subset V(G)$ violates
  \textbf{subset}, return $x$, which is the inventory of a
  minimum-weight quota forest.

\item Otherwise, form $G'=G/S$, and define $w'(e)=w(e)$ unless $e$
  is an edge from $G\setminus S$ to $S$, in which case set
  $w'(e)=w(e) - w(e')$, where $e'$ is the heaviest edge in $G[S][x]$
  such that $t(e)=t(e')$. Set $q'(v_S)=1$ and $s'(v_S)=0$,
  where $v_S$ is the new vertex formed by contracting $S$; otherwise
  $q'=q$ and $s'=s$. Finally, for each edge $e$ from $S$ to
  $V(G)\setminus S$, set $c'_e = c_e q(i(e))$.
  Note that if $|S|>1$, we have decreased the number of vertices
  without adding any loops; if $|S|=1$ the number of vertices is
  unchanged but we have deleted any loops over $S$. So we can proceed
  by induction on the number of vertices plus the number of loops.

\item Apply the algorithm recursively to compute the inventory of a
  MQF $T'$ for $(G',q',s',w')$. Then the inventory of a MQF $T$ for
  $(G,q,s,w)$ is $w(T)=w(T')+w(C)$, where
  $$w(C)=\sum_{v\in S}q(v)=\sum_{e\in E(G[S])}x_e$$
  is the weight of the Eulerian cycle $C$ which immerses onto
  $G[S][x]$ in the proof of Theorem \ref{inventoryconstraints}.

\item If desired, construct $T'$ explicitly and extend it to a
  quota forest $T$ for $G$. As in Edmonds' algorithm, we start by
  setting $T=T'\cup C\setminus \{e'\}$. But we now have to
  redistribute the edges coming out of $v_S$: for any edge
  $e\in E(G)$ from a vertex $v\in S$ to a vertex in $V(S)\setminus S$,
  $T'$ may contain up to $q(v)c_e$ copies of $e$. The initial points
  of these edges can be distributed arbitrarily among the $q(v)$
  lifts of $v$ in $C$, so that no lift has more than $c_e$ copies of
  $e$ coming out. The resulting $T$ is a MQF for $(G,q,s,w)$.

\end{enumerate}
\end{algorithm}

We remark that all inventories of MSAs are vertices of the
spanning-arborescence polyhedron, since they are all $\{0,1\}$ vectors
with a fixed norm.  This is no longer true for MQFs; in particular, if
a vertex $v$ has two incoming edges with the same weight, these edges
can be traded off for each other, and only the extreme choices can be
vertices of the inventory polyhedron.

If we only care about finding a minimum inventory of an MQF, without
actually trying to reconstruct the forest itself, the complexity of a
na\"ive implementation of this algorithm is essentially the same as
that of Edmonds' algorithm. The main difference is that, instead of
finding the least-weight edge into each vertex $v$, we now need to
find the $q(v)-s(v)$ edges of least weight; but we can preserve the
$O(E)$ complexity of this step by using a linear-time order-statistic
algorithm.  Our algorithm takes at most $O(V)$ steps, so accounting
for the light quota arithmetic needed, we get a complexity of
$O(EV\log ||q||_\infty)$.

\subsubsection*{Example}

This example highlights some complications which do not arise in
Edmonds' algorithm, such as edge replication, the relevance of loops,
and keeping track of quotas. Consider the directed graph $G$ in Figure
\ref{fig:MQTexample1}, whose edges have weights ranging from $1$ to
$4$. Quotas are indicated on the vertices; the start portfolio
consists of one copy of the quota-$4$ vertex.

\newdimen\linewd \linewd=0.6pt
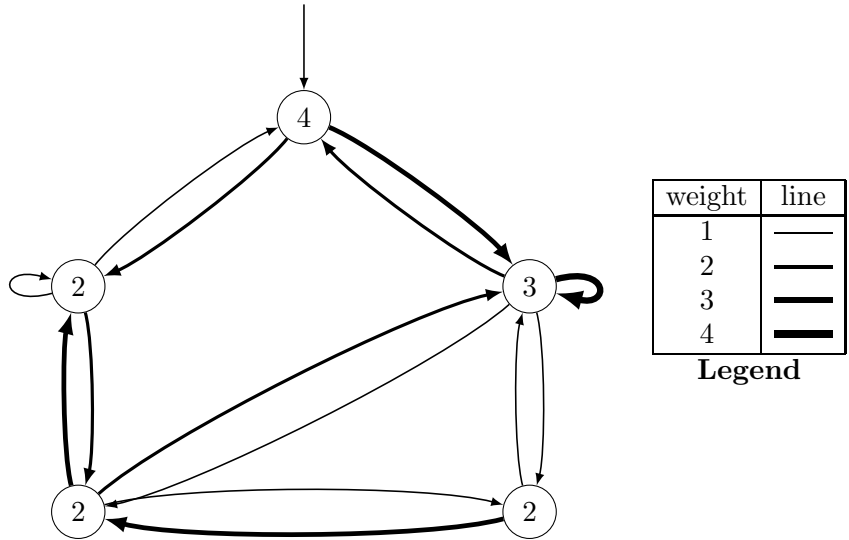
\begin{figure}
\begin{tikzpicture}[scale=1.5]
\foreach \i/\q/\x/\y in {1/4/2/3.5, 2/2/0/2, 3/2/0/0, 4/2/4/0, 5/3/4/2}
  \node at (\x,\y) [circle, inner sep=4pt, draw] (P\i) {$\q$};
\coordinate (P0) at (2,4.5);

\foreach \i/\j/\w in { 
  1/2/2, 2/3/2, 3/4/1, 4/5/1, 5/1/2,
  1/5/3, 5/4/1, 4/3/3, 3/2/3, 2/1/1,
  3/5/2, 5/3/1
}
\draw[->, line width=\w\linewd] (P\i) to 
  [bend left=15, looseness=0.5, sloped] (P\j);

\draw[->, line width=1\linewd] (P0) to (P1);
\draw[->, line width=1\linewd] (P2) to [loop left] ();
\draw[->, line width=4\linewd] (P5) to [loop right] ();

\node[anchor=south west] at (5,1) {
  \begin{tabular}{|c|c|}
  \hline
  weight & line\\ 
  \hline
  1 & \rule[0.5ex]{2em}{1\linewd}\\
  2 & \rule[0.5ex]{2em}{2\linewd}\\
  3 & \rule[0.5ex]{2em}{3\linewd}\\
  4 & \rule[0.5ex]{2em}{4\linewd}\\
  \hline
  \multicolumn{2}{c}{\textbf{Legend}}
  \end{tabular}
};
\end{tikzpicture}
\caption{A directed, edge-weighted graph for which we will compute a
  minimum quota tree.  Quotas are indicated on the vertices; edge
  weights range from $1$ to $4$, indicated by line thickness.}
\label{fig:MQTexample1}
\end{figure}

The first step in the algorithm greedily selects a minimum edge inventory
subject to the \textbf{edge} and \textbf{node} constraints. See Figure
\ref{fig:MQTexample2}; each edge $e$ is labeled by $x_e$, the number of
copies of that edge in the inventory. In step $2$ we find two subsets
$S$ violating the \textbf{subset} condition; these are indicated with
dashed lines in Figure \ref{fig:MQTexample2}.

\begin{figure}
\begin{tikzpicture} [scale=1.5]
\foreach \i/\q/\x/\y in {1/4/2/3.5, 2/2/0/2, 3/2/0/0, 4/2/4/0, 5/3/4/2} {
  \node at (\x,\y) [circle, inner sep=4pt, draw] (P\i) {$\q$};
  \node at (\x,\y) [circle, minimum size=60pt]   (C\i) {}; 
}
\coordinate (P0) at (2,4.5);

\foreach \i/\j/\w/\c in { 
  3/4/1/1, 4/5/1/2, 5/1/2/1, 5/4/1/1, 2/1/1/2, 3/5/2/1, 5/3/1/2
}
  \draw[->, line width=\w\linewd] (P\i)
     to [bend left=15, looseness=0.5, sloped]
     node[midway, fill=white, opaque] {$\c$}
     (P\j);

\draw[->, line width=1\linewd] (P0) to 
     (P1);

\draw[->, line width=1\linewd] (P2) to [loop left] node[auto=left] {2} ();

\draw[blue,dashed] (P2) circle (20pt);



\draw (C3.270) [red,dashed] -- (C4.270) arc (270:360:20pt) -- (C5.0)
   arc (0:116.5:20pt) -- (C3.116.5) arc (116.5:270:20pt) -- cycle;

\end{tikzpicture}
\caption{The multigraph $G[x]$ representing the minimal-weight
  inventory satisfying the \textbf{edge} and \textbf{node}
  conditions. Edge $e$ is labeled with the inventory value $x_e$ Two
  clusters violate the \textbf{subset} constraint; these are sources
  in the strong-component graph of the inventory.}
\label{fig:MQTexample2}
\end{figure}
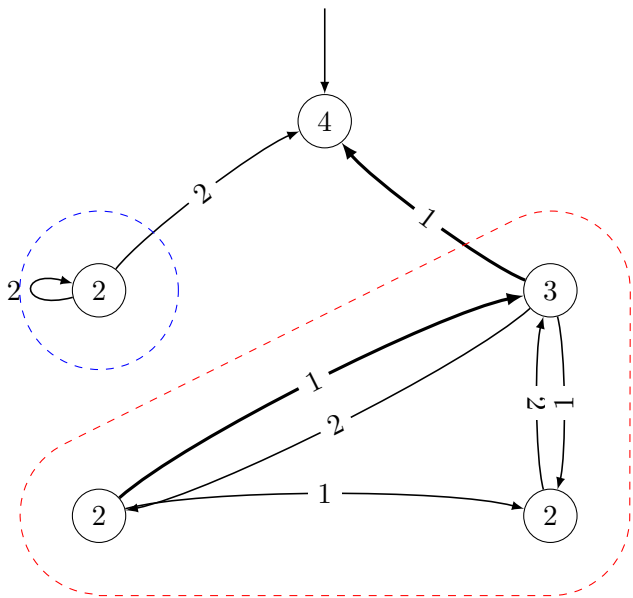

In step 3 we form a quotient graph by collapsing each violating
subset to a single node with quota $1$, removing internal edges. (Note
when we ``collapse'' the singleton subset, we still have a that one
subset is a singleton with a loop; the collapse removes the loop.)
Edges out of the subset are reweighted and duplicated as
described. The resulting quotient graph $G'$, with new quotas $q'$,
is shown in Figure
\ref{fig:MQTexample3}.  The second number $c$ on each edge indicates
that $G'$ actually contains $c$ copies of that edge.  For future
reference, we remember for each
subset the weight of the associated immersed Eulerian cycle; in this
example, that is $2$ for the singleton subset and $8$ for the
$3$-vertex subset.

In step 4 we recursively apply the algorithm to $(G',q',s',w')$.
In this case there
are no more \textbf{subset} violations, so the minimal inventory comes from
a quota tree. The first number on each edge in Figure \ref{fig:MQTexample3}
indicates the number of copies of that edge in the inventory.  The weight
of this inventory is $6$, which is the weight of a MQT for $G'$. To get the
weight of a MQT for $G$, we need to add in the weights of the Eulerian
cycles we collapsed; thus a MQT for $G$ has weight $6+2+8=16$. Figure
\ref{fig:MQTexample4} shows the weight-$16$ MQT for $G$ constructed by
step $5$ of the algorithm.

\begin{figure}
\begin{tikzpicture} [scale=1.5]
\foreach \i/\q/\x/\y/\c in {1/4/2/3.5/black, 2/1/0.3/1/blue, 3/1/3.7/1/red}
  \node at (\x,\y) [circle, inner sep=4pt, draw=\c] (P\i) {$\q$};
\coordinate (P0) at (2,4.5);
\draw[->, line width=1\linewd] (P0) to 
     (P1);

\foreach \i/\j/\w/\u/\c in { 
  1/2/1/1/1, 2/1/1/2/2, 1/3/1/1/1, 3/1/2/1/3, 2/3/1/0/2, 3/2/2/0/2
}
  \draw[->, line width=\w\linewd] (P\i)
     to [bend left=15, looseness=0.5, sloped]
     node[midway, fill=white, opaque] {$\u/\c$}
     (P\j);

\draw[->, line width=2\linewd] (P3) to [loop right] 
   node [midway] {0/3} ();
\end{tikzpicture}
\caption{The quotient graph $G'$ obtained by collapsing the subsets
  $S$, and reweighting and duplicating the edges out of the collapsed
  nodes.  New weights are shown according to the legend in Figure
  \ref{fig:MQTexample1}.  The label $x_e/c_e$ on an edge $e$ indicates
  that $c_e$ copies of that edge are available, of which $x_e$ copies
  are used in the minimal inventory.}
\label{fig:MQTexample3}
\end{figure}
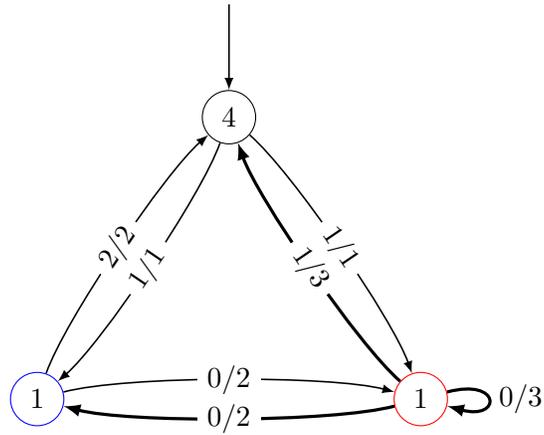

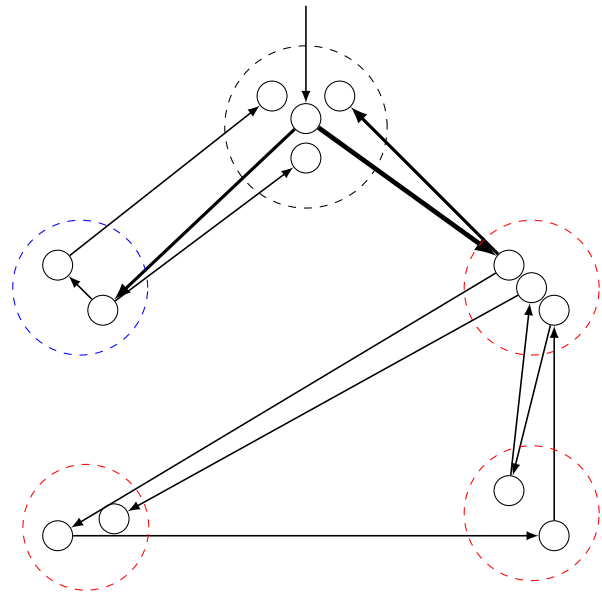
\begin{figure}
\begin{tikzpicture} [scale=1.5]
\foreach \label/\x/\y in {
  1a/2/3.5, 1b/1.7/3.7, 1c/2.3/3.7, 1d/2/3.15,
  2a/-0.2/2.2, 2b/0.2/1.8,
  3a/-0.2/-0.2, 3b/0.3/-0.05,
  4a/3.8/0.2, 4b/4.2/-0.2,
  5a/3.8/2.2, 5b/4/2, 5c/4.2/1.8
}
  \node at (\x,\y) [circle, inner sep=4pt, draw] (\label) {};
\coordinate (P0) at (2,4.5);

\foreach \i/\j/\w in {
  P0/1a/1 ,
  1a/2b/2, 2b/1d/1, 2b/2a/1, 2a/1b/1,
  1a/5a/3, 5a/1c/2,
  5a/3a/1, 3a/4b/1, 4b/5c/1, 5c/4a/1, 4a/5b/1, 5b/3b/1
}
\draw[->, line width=\w\linewd] (\i) -- (\j);

\node [draw=black, dashed, circle, fit=(1a) (1b) (1c) (1d)] {};
\node [draw=blue,  dashed, circle, fit=(2a) (2b)          ] {};
\node [draw=red,   dashed, circle, fit=(3a) (3b)          ] {};
\node [draw=red,   dashed, circle, fit=(4a) (4b)          ] {};
\node [draw=red,   dashed, circle, fit=(5a) (5b) (5c)     ] {};

\end{tikzpicture}
\caption{The MQT for the graph $G$ as constructed by this algorithm; it has
weight 16.}
\label{fig:MQTexample4}
\end{figure}

\section{Further work}
\label{sec:further-work}
Some open problems/research directions related to quota trees:

\begin{enumerate}

\item Lagrange inversion is powerful but sometimes unsatisfying.  For
  example, Goulden \cite[\S3.3.10]{goulden2004combinatorial} uses it
  to give a one-line proof of Cayley's result that there are $n^{n-1}$
  labeled rooted trees on $n$ nodes, but the proof is somewhat
  unsatisfying since it doesn't yield a combinatorial correspondence
  between trees and sequences.  Give a similar bijective proof of
  Theorem \ref{countingforests}, for example using objects such as
  Pr\"ufer sequences or parking functions.

\item Quota search provides an illuminating context for the
  all-dest\-in\-a\-tion $k$-lightest-paths problem, but the generic
  meta-algorithm as presented here treats the priority queue as a
  black box, and is thus less efficient than (for example) Eppstein's
  algorithm \cite{eppstein1998finding}. Can existing $k$-lightest-path
  algorithms be formulated in terms of quota search?

\item Extend the combinatorial reciprocity theorem for roses to
  general quota trees. That is, give a combinatorial interpretation of
  the objects counted by quota forests over more general graphs with a
  formally negative start portfolio.

\item Another approach to uniformly sampling a spanning tree of a
  graph $G$ is to use random walks on $G$.  Wilson's algorithm (see
  \cite{wilson1996generating}) in particular is very tempting in this
  context, as it is simple, efficient, and applies to directed
  graphs. It would be interesting to adapt such a random-walk-based
  algorithm to the context of quota trees.

\item Tarjan \cite{tarjan1977finding} gives an implementation of
  Edmonds' algorithm with complexity $O(E\log V)$ for sparse graphs,
  or $O(V^2)$ for dense graphs. Adapt these algorithms to the
  context of finding minimum-weight quota forests.

\item Identify applications where quota trees arise naturally
  (e.g. epidemiology, message broadcast, etc.) Find appearances
  of quota trees in other guises (e.g. the Narayana numbers),
  develop dictionaries relating other areas of mathematics to
  quota trees, and drive the theory accordingly.

\end{enumerate}

\bibliographystyle{plain}
\bibliography{quotatrees}

\end{document}